\documentclass[11pt,letterpaper]{article}
\usepackage{epsfig}
\usepackage{amssymb}
\usepackage{amsthm}
\usepackage{amsmath}
\usepackage[margin=1in]{geometry}

\newtheorem{theorem}{Theorem}[section]
\newtheorem{lemma}[theorem]{Lemma}

\newtheorem{proposition}[theorem]{Proposition}

\newtheorem{definition}[theorem]{Definition}
\newtheorem{remark}[theorem]{Remark}

\newtheorem{problem}{Problem}
\newtheorem{solution}{Solution}

\newcommand{\R}{{\mathbb R}}

\newcommand{\ones}{\mathbf 1}

\newcommand{\one}{\mathbf{1}}
\newcommand{\Sn}{\mathbf{S}^n}
\newcommand{\e}{\mathbf{e}}
\newcommand{\by}{\mathbf{y}}
\newcommand{\bd}{\mathbf{d}}
\newcommand{\A}{\mathcal{A}}
\newcommand{\bx}{\mathbf{x}}
\newcommand{\bz}{\mathbf{z}}
\newcommand{\bq}{\mathbf{q}}
\newcommand{\N}{\mathcal{N}}
\newcommand{\T}{\mathcal{T}}

\newcommand{\tr}{\mathop{\mathrm{Tr}}}

\begin{document}
\title{Convex Graph Invariants}

\author{Venkat Chandrasekaran, Pablo A. Parrilo, and Alan S. Willsky \thanks{Email: \{venkatc,parrilo,willsky\}@mit.edu.
This work was supported in part by AFOSR grant FA9550-08-1-0180, in part by a MURI through ARO grant W911NF-06-1-0076, in part by a MURI through AFOSR grant FA9550-06-1-0303, and in part by NSF FRG 0757207. } \vspace{0.25in} \\ Laboratory for Information and Decision Systems \\ Department of Electrical Engineering and Computer Science \\ Massachusetts Institute of Technology \\ Cambridge, MA 02139 USA}
\date{\today}

\maketitle

\begin{abstract}
%
%
%
%
%
%
%
%

The structural properties of graphs are usually characterized in terms of invariants, which are functions of graphs that do not depend on the labeling of the nodes. In this paper we study convex graph invariants, which are graph invariants that are convex functions of the adjacency matrix of a graph.  Some examples include functions of a graph such as the maximum degree, the MAXCUT value (and its semidefinite relaxation), and spectral invariants such as the sum of the $k$ largest eigenvalues.  Such functions can be used to construct convex sets that impose various structural constraints on graphs, and thus provide a unified framework for solving a number of interesting graph problems via convex optimization.  We give a representation of all convex graph invariants in terms of certain elementary invariants, and describe methods to compute or approximate convex graph invariants tractably.  We also compare convex and non-convex invariants, and discuss connections to robust optimization.  Finally we use convex graph invariants to provide efficient convex programming solutions to graph problems such as the deconvolution of the composition of two graphs into the individual components, hypothesis testing between graph families, and the generation of graphs with certain desired structural properties.


\end{abstract}

\section{Introduction}
\label{sec:intro}

Graphs are useful in many applications throughout science and engineering as they offer a concise model for relationships among a large number of interacting entities.  These relationships are often best understood using structural properties of graphs.  \emph{Graph invariants} play an important role in characterizing abstract structural features of a graph, as they do not depend on the labeling of the nodes of the graph.  Indeed families of graphs that share common structural attributes are often specified via graph invariants.  For example bipartite graphs can be defined by the property that they contain no cycles of odd length, while the family of regular graphs consists of graphs in which all nodes have the same degree.  Such descriptions of classes of graphs in terms of invariants have found applications in areas as varied as combinatorics \cite{Die2005}, network analysis in chemistry \cite{Bon1991} and in biology \cite{MasV2007}, and in machine learning \cite{Lau1996}.  For instance the treewidth \cite{RobS1984} of a graph is a basic invariant that governs the complexity of various algorithms for graph problems.

We begin by introducing three canonical problems involving structural properties of graphs, and the development of a unified solution framework to address these questions serves as motivation for our discussion throughout this paper.
\begin{itemize}
\item{\bf Graph deconvolution.} Suppose we are given a graph that is the combination of two known graphs overlaid on the same set of nodes.  How do we recover the individual components from the composite graph?  For example in Figure~\ref{fig:dec1} we are given a composite graph that is formed by adding a cycle and the Clebsch graph.  Given no extra knowledge of any labeling of the nodes, can we ``deconvolve'' the composite graph into the individual cycle/Clebsch graph components?

\item{\bf Graph generation.} Given certain structural constraints specified by invariants how do we produce a graph that satisfies these constraints?  A well-studied example is the question of constructing expander graphs.  Another example may be that we wish to recover a graph given constraints, for instance, on certain subgraphs being forbidden, on the degree distribution, and on the spectral distribution.

\item{\bf Graph hypothesis testing.} Suppose we have two families of graphs, each characterized by some common structural properties specified by a set of invariants;  given a new sample graph which of the two families offers a ``better explanation'' of the sample graph (see Figure~\ref{fig:ht})?
\end{itemize}
In Section~\ref{sec:apps} we describe these problems in more detail, and also give some concrete applications in network analysis and modeling in which such questions are of interest.

To efficiently solve problems such as these we wish to develop a collection of tractable computational tools.  Convex relaxation techniques offer a candidate framework as they possess numerous favorable properties.  Due to their powerful modeling capabilities, convex optimization methods can provide tractable formulations for solving difficult combinatorial problems exactly or approximately.  Further convex programs may often be solved effectively using general-purpose off-the-shelf software.  Finally one can also give conditions for the success of these convex relaxations based on standard optimality results from convex analysis.

Motivated by these considerations we introduce and study \emph{convex graph invariants} in Section~\ref{sec:cgi}.  These invariants are convex functions of the adjacency matrix of a graph.  More formally letting $A$ denote the adjacency matrix of a (weighted) graph, a convex graph invariant is a convex function $f$ such that $f(A) = f(\Pi A \Pi^T)$ for all permutation matrices $\Pi$.  Examples include functions of a graph such as the maximum degree, the MAXCUT value (and its semidefinite relaxation), the second smallest eigenvalue of the Laplacian (a concave invariant), and spectral invariants such as the sum of the $k$ largest eigenvalues; see Section~\ref{subsec:excgi} for a more comprehensive list.  As some of these invariants may possibly be hard to compute, we discuss in the sequel the question of approximating intractable convex invariants.  We also study \emph{invariant convex sets}, which are convex sets with the property that a symmetric matrix $A$ is a member of such a set if and only if $\Pi A \Pi^T$ is also a member of the set for all permutations $\Pi$.  Such convex sets are useful in order to impose various structural constraints on graphs.  For example invariant convex sets can be used to express forbidden subgraph constraints (i.e., that a graph does not contain a particular subgraph such as a triangle), or require that a graph be connected; see Section~\ref{subsec:exics} for more examples.  We compare the strengths and weaknesses of convex graph invariants versus more general non-convex graph invariants.  Finally we also provide a robust optimization perspective of invariant convex sets.  In particular we make connections between our work and the data-driven perspective on robust optimization studied in \cite{BerB2009}.


In order to systematically evaluate the expressive power of convex graph invariants we analyze \emph{elementary} convex graph invariants, which serve as a basis for constructing arbitrary convex invariants.  Given a symmetric matrix $P$, these elementary invariants (again, possibly hard to compute depending on the choice of $P$) are defined as follows:
\begin{equation}
\Theta_P(A) = \max_{\Pi} \tr(P \Pi A \Pi^T), \label{eq:introecgi}
\end{equation}
where $A$ represents the adjacency matrix of a graph, and the maximum is taken over all permutation matrices $\Pi$.  It is clear that $\Theta_P$ is a convex graph invariant, because it is expressed as the maximum over a set of linear functions.  Indeed several simple convex graph invariants can be expressed using functions of the form \eqref{eq:introecgi}.  For example $P = I$ gives us the total sum of the node weights, while $P = \one \one^T - I$ gives us twice the total (weighted) degree.  Our main theoretical results in Section~\ref{sec:cgi} can be summarized as follows:  First we give a representation theorem stating that any convex graph invariant can be expressed as the supremum over elementary convex graph invariants \eqref{eq:introecgi} (see Theorem~\ref{theo:cgirep}).  Second we have a similar result stating that any invariant convex set can be expressed as the intersection of convex sets given by level sets of the elementary invariants \eqref{eq:introecgi} (see Proposition~\ref{theo:icsrep}).  These results follow as a consequence of the separation theorem from convex analysis.  Finally we also show that for any two non-isomorphic graphs given by adjacency matrices $A_1$ and $A_2$, there exists a $P$ such that $\Theta_P(A_1) \neq \Theta_P(A_2)$ (see Lemma~\ref{theo:complete}).  Hence convex graph invariants offer a \emph{complete} set of invariants as they can distinguish between non-isomorphic graphs.

In Section~\ref{subsec:spec} we discuss an important subclass of convex graph invariants, namely the set of convex \emph{spectral invariants}.  These are convex functions of symmetric matrices that depend only on the eigenvalues, and can equivalently be expressed as the set of convex functions of symmetric matrices that are invariant under conjugation by orthogonal matrices (note that convex graph invariants are only required to be invariant with respect to conjugation by permutation matrices) \cite{Dav1957}. The properties of convex spectral invariants are well-understood, and they are useful in a number of practically relevant problems (e.g., characterizing the subdifferential of a unitarily invariant matrix norm \cite{Wat92}).  These invariants play a prominent role in our experimental demonstrations in Section~\ref{sec:appscgi}.

As noted above convex graph invariants, and even elementary invariants, may in general be hard to compute.  In Section~\ref{sec:comp} we investigate the question of approximately computing these invariants in a tractable manner.  For many interesting special cases such as the MAXCUT value of a graph, or (the inverse of) the stability number, there exist well-known tractable semidefinite programming (SDP) relaxations that can be used as surrogates instead \cite{GoeW1995,MotS1965}.  More generally functions of the form of our elementary convex invariants \eqref{eq:introecgi} have appeared previously in the literature; see \cite{Cel1998} for a survey.  Specifically we note that evaluating the function $\Theta_P(A)$ for any fixed $A,P$ is equivalent to solving the so-called Quadratic Assignment Problem (QAP), and thus we can employ various tractable linear programming, spectral, and SDP relaxations of QAP \cite{ZhaKRW1998,Cel1998,RenS2007}.  In particular we discuss recent work \cite{deKS2010} on exploiting group symmetry in SDP relaxations of QAP, which is useful for approximately computing elementary convex graph invariants in many interesting cases.

Finally in Section~\ref{sec:appscgi} we return to the motivating problems described previously, and give solutions to these questions. These solutions are based on convex programming formulations, with convex graph invariants playing a fundamental role.  We give theoretical conditions for the success of these convex formulations in solving the problems discussed above, and experimental demonstration for their effectiveness in practice.  Indeed the framework provided by convex graph invariants allows for a \emph{unified} investigation of our proposed solutions.  As an example result we give a tractable convex program (in fact an SDP) in Section \ref{subsec:dec} to ``deconvolve'' the cycle and the Clebsch graph from a composite graph consisting of these components (see Figure~\ref{fig:dec1}); a salient feature of this convex program is that it only uses \emph{spectral invariants} to perform the decomposition.

\paragraph{Summary of contributions}  We emphasize again the main contributions of this paper.  We begin by introducing three canonical problems involving structural properties of graphs.  These problems arise in various applications (see Section~\ref{sec:apps}), and serve as a motivation for our discussion in this paper.  In order to solve these problems we introduce convex graph invariants, and investigate their properties (see Section~\ref{sec:cgi}).  Specifically we provide a representation theorem of convex graph invariants in terms of elementary invariants, and we make connections between these ideas and concepts from other areas such as robust optimization.  Finally we describe tractable convex programming solutions to the motivating problems based on convex graph invariants (see Section~\ref{sec:appscgi}).  Therefore, convex graph invariants provide a useful computational framework based on convex optimization for graph problems.

\paragraph{Related previous work}  We note that convex optimization methods have been used previously to solve various graph-related problems.  We would particularly like to emphasize a body of work on convex programming formulations to optimize convex functions of the Laplacian eigenvalues of graphs \cite{BoyDX2004,Boy2006} subject to various constraints.  Although our objective is similar in that we seek solutions based on convex optimization to graph problems, our work is different in several respects from these previous approaches.  While the problems discussed in \cite{Boy2006} explicitly involved the optimization of spectral functions, other graph problems such as those described in Section~\ref{sec:apps} may require non-spectral approaches (for example, hypothesis testing between two families of graphs that are isospectral, i.e., have the same spectrum, but are distinguished by other structural properties).  As convex spectral invariants form a subset of convex graph invariants, the framework proposed in this paper offers a larger suite of convex programming methods for graph problems.  More broadly our work is the first to formally introduce and characterize convex graph invariants, and to investigate their properties as natural mathematical objects of independent interest.

\paragraph{Outline}  In Section~\ref{sec:apps} we give more details of the questions that motivate our study of convex graph invariants.  Section~\ref{sec:cgi} gives the definition of convex graph invariants and invariant convex sets, as well as several examples of these such functions and sets.  We also discuss various properties of convex graph invariants in this section.  In Section~\ref{sec:comp} we investigate the question of efficiently computing approximations to intractable convex graph invariants.  We give detailed solutions using convex graph invariants to each of our motivating problems in Section~\ref{sec:appscgi}, and we conclude with a brief discussion in Section~\ref{sec:conc}.



\section{Applications}
\label{sec:apps}

In this section we describe three problems involving structural properties of graphs, which serve as a motivation for our investigation of convex graph invariants.  In Section~\ref{sec:appscgi} we give solutions to these problems using convex graph invariants.

\subsection{Graph Deconvolution}
\label{subsec:appsdec}

Suppose we are given a graph that is formed by overlaying two graphs on the same set of nodes.  More formally we have a graph whose adjacency matrix is formed by adding the adjacency matrices of two known graphs.  However, we do not have any information about the relative labeling of the nodes in the two component graphs.  Can we recover the individual components from the composite graph?  As an example suppose we are given the combination of a cycle and a grid, or a cycle and the Clebsch graph, on the same set of nodes.  Without any additional information about the labeling of the nodes, which may reveal the cycle/grid or cycle/Clebsch graph structure, the goal is to recover the individual components. Figure~\ref{fig:dec1} gives a graphical illustration of this question.  In general such decomposition problems may be ill-posed, and it is of interest to give conditions under which unique deconvolution is possible as well as to provide tractable computational methods to recover the individual components.  In Section~\ref{subsec:dec} we describe an approach based on convex optimization for graph deconvolution;  for example this method decomposes the cycle and the Clebsch graph from a composite graph consisting of these components (see Figure~\ref{fig:dec1}) using only the spectral properties of the two graphs.

Well-known problems that have the flavor of graph deconvolution include the \emph{planted clique} problem, which involves identifying hidden cliques embedded inside a larger graph, and the \emph{clustering} problem in which the goal is to decompose a large graph into smaller densely connected clusters by removing just a few edges.  Convex optimization approaches for solving such problems have been proposed recently \cite{AmeV2009,AmeV2010}.  Graph deconvolution more generally may include other kinds of embedded structures beyond cliques.

Applications of graph deconvolution arise in network analysis in which one seeks to better understand a complex network by decomposing it into simpler components.  Graphs play an important role in modeling, for example, biological networks \cite{MasV2007} and social networks \cite{Jac2008,EasK2010}, and lead to natural graph deconvolution problems in these areas.  For instance graphs are useful for describing social exchange networks of interactions of multiple agents, and graph decompositions are useful for describing the structure of optimal bargaining solutions in such networks \cite{KleT2008}.  In a biological network setting, transcriptional regulatory networks of bacteria have been observed to consist of small subgraphs with specific structure (called motifs) that are connected together using a ``backbone'' \cite{DobBBO2004}.  Decomposing such regulatory networks into the component structures is useful for obtaining a better understanding of the high-level properties of the composite network.

\begin{figure}
\begin{center}
\epsfig{file=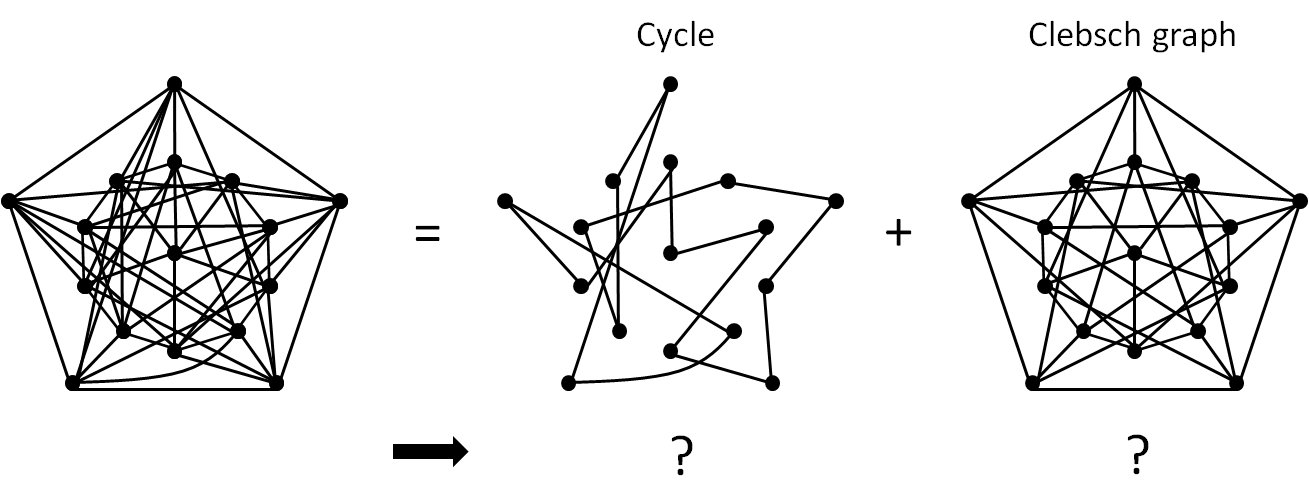,width=12cm,height=5cm} \caption{An instance of a deconvolution problem: Given a composite graph formed by adding the $16$-cycle and the Clebsch graph, we wish to recover the individual components.  The Clebsch graph is an example of a strongly regular graph on $16$ nodes \cite{GodR2004}; see Section~\ref{subsec:dec} for more details about the properties of such graphs.} \label{fig:dec1}
\end{center}
\end{figure}

\subsection{Generating Graphs with Desired Structural Properties}
\label{subsec:appsid}

Suppose we wish to construct a graph with certain prescribed structural constraints.  A very simple example may be the problem of constructing a graph in which each node has degree equal to two.  A graph given by a single cycle satisfies this constraint.  A less trivial problem is one in which the objective may be to build a connected graph with constraints on the spectrum of the adjacency matrix, the degree distribution, and the additional requirements that the graph be triangle-free and square-free.  Of course such graph reconstruction problems may be infeasible in general, as there may be no graph consistent with the given constraints. Therefore it is of interest to derive suitable conditions under which this problem may be well-posed, and to develop a suitably flexible yet tractable computational framework to incorporate any structural information available about a graph.

A prominent instance of a graph construction problem that has received much attention is the question of generating expander graphs \cite{HooLW2006}.  Expanders are, roughly speaking, sparse graphs that are well-connected, and they have found applications in numerous areas of computer science.  Methods used to construct expanders range from random sampling approaches to deterministic constructions based on Ramanujan graphs \cite{HooLW2006}.  In Section~\ref{subsec:id} we describe an approach based on convex optimization to generate sparse, weighted graphs with small degree and large spectral gap.

\begin{figure}
\begin{center}
\epsfig{file=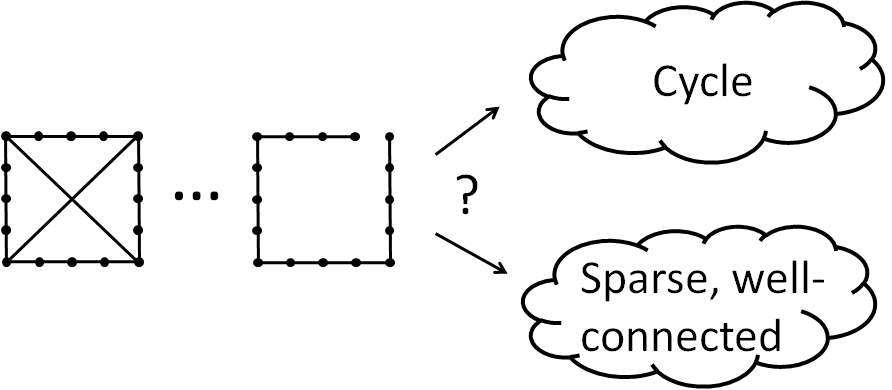,width=10cm,height=5cm} \caption{An instance of a hypothesis testing problem: We wish to decide which family of graphs offers a ``better explanation'' for a given candidate sample graph.} \label{fig:ht}
\end{center}
\end{figure}

\subsection{Graph Hypothesis Testing}
\label{subsec:appsht}

As our third problem we consider a more statistically motivated question.  Suppose we have two families of graphs each characterized by some common structural properties specified by certain invariants.  Given a new sample graph which of these two families offers a ``better explanation'' for the sample graph?  For example as illustrated in Figure~\ref{fig:ht} we may have two families of graphs -- one being the collection of cycles, and the other being the set of sparse, well-connected graphs.  If a new sample graph is a path (i.e., a cycle with an edge removed), we would expect that the family of cycles should be a better explanation.  On the other hand if the sample is a cycle plus some edges connecting diametrically opposite nodes, then the second family of sparse, well-connected graphs offers a more plausible fit.  Notice that these classes of graphs may often be specified in terms of \emph{different} sets of invariants, and it is of interest to develop a suitable framework in which we can incorporate diverse structural information provided about graph families.

We differentiate this problem from the well-studied question of \emph{testing properties} of graphs \cite{GolGR1996}.  Examples of property testing include testing whether a graph is $3$-colorable, or whether it is close to being bipartite.  An important goal in property testing is that one wishes to test for graph properties by only making a small number of ``queries'' of a graph.  We do not explicitly seek such an objective in our algorithms for hypothesis testing.  We also note that hypothesis testing can be posed more generally than a yes/no question as in property testing, and as mentioned above the two families in hypothesis testing may be specified in terms of very different sets of invariants.

In order to address the hypothesis testing question in a statistical framework, we would need a statistical theory for graphs and appropriate error metrics with respect to which one could devise optimal decision rules.  In Section~\ref{subsec:ht} we discuss a computational approach to the hypothesis testing problem using convex graph invariants that gives good empirical performance, and we defer the issue of developing a formal statistical framework to future work.


\section{Convex Graph Invariants}
\label{sec:cgi}

In this section we define convex graph invariants, and discuss their properties.  Throughout this paper we denote the space of $n \times n$ symmetric matrices by $\Sn \simeq \mathbb{R}^{n+1 \choose 2}$.  All our definitions of convexity are with respect to the space $\Sn$.  We consider undirected graphs that do not have multiple edges and no self-loops; these are represented by adjacency matrices that lie in $\Sn$.  Therefore a graph may possibly have node weights and edge weights.  A graph is said to be \emph{unweighted} if its node weights are zero, and if each edge has a weight of one (non-edges have a weight of zero); otherwise a graph is said to be \emph{weighted}.  Let $\e_i \in \R^n$ denote the vector with a one in the $i$'th entry and zero elsewhere, let $I$ denote the $n \times n$ identity matrix, let $\ones \in \R^n$ denote the all-ones vector, and let $J = \ones \ones^T \in \Sn$ denote the all-ones matrix.  Further we let $\A = \{A ~:~ A \in \Sn, ~ 0 \leq A_{i,j} \leq 1 ~ \forall i,j\}$; we will sometimes find it useful in our examples in Section~\ref{subsec:exics} to restrict our attention to graphs with adjacency matrices in $\A$.  Next let $\mathrm{Sym}(n)$ denote the symmetric group over $n$ elements, i.e., the group of permutations of $n$ elements.  Elements of this group are represented by $n \times n$ permutation matrices.  Let $\mathrm{O}(n)$ represent the orthogonal group of $n \times n$ orthogonal matrices.  Finally given a vector $\bx \in \R^n$ we let $\overline{\bx}$ denote the vector obtained by sorting the entries of $\bx$ in descending order.

\subsection{Motivation: Graphs and Adjacency Matrices}
\label{subsec:mot}

Matrix representations of graphs in terms of adjacency matrices and Laplacians have been used widely both in applications as well as in the analysis of the structure of graphs based on algebraic properties of these matrices \cite{Big1994}.  For example the spectrum of the Laplacian of a graph reveals whether a graph is ``diffusive'' \cite{HooLW2006}, or whether it is even connected.  The degree sequence, which may be obtained from the adjacency matrix or the Laplacian, reveals whether a graph is regular, and it plays a role in a number of real-world investigations of graphs arising in social networks and the Internet.

Given a graph $\mathcal{G}$ defined on $n$ nodes, a \emph{labeling} of the nodes of $\mathcal{G}$ is a function $\ell$ that maps the nodes of $\mathcal{G}$ onto distinct integers in $\{1,\dots,n\}$.  An adjacency matrix $A \in \Sn$ is then said to \emph{represent} or \emph{specify} $\mathcal{G}$ if there exists a labeling $\ell$ of the nodes of $\mathcal{G}$ so that the weight of the edge between nodes $i$ and $j$ equals $A_{\ell(i) \ell(j)}$ for all pairs $\{i,j\}$ and the weight of node $i$ equals $A_{\ell(i) \ell(i)}$ for all $i$.  However an adjacency matrix representation $A$ of the graph $\mathcal{G}$ is not unique.  In particular $\Pi A \Pi^T$ also specifies $\mathcal{G}$ for all $\Pi \in \mathrm{Sym}(n)$.  All these alternative adjacency matrices correspond to different labelings of the nodes of $\mathcal{G}$. \emph{Thus the graph $\mathcal{G}$ is specified by the matrix $A$ only up to a relabeling of the indices of $A$.}  Our objective is to describe abstract structural properties of $\mathcal{G}$ that do \emph{not} depend on a choice of labeling of the nodes.  In order to characterize such \emph{unlabeled} graphs in which the nodes have no distinct identity except through their connections to other nodes, it is important that any function of an adjacency matrix representation of a graph not depend on the particular choice of indices of $A$.  Therefore we seek functions of adjacency matrices that are invariant under conjugation by permutation matrices, and denote such functions as \emph{graph invariants}.

\subsection{Definition of Convex Invariants}
\label{subsec:cgidef}

A convex graph invariant is an invariant that is a convex function of the adjacency matrix of a graph.  Specifically we have the following definition:

\begin{definition}
A function $f: \Sn \rightarrow \R$ is a \emph{convex graph invariant} if it is convex, and if for any $A \in \Sn$ it holds that $f(\Pi A \Pi^T) = f(A)$ for all permutation matrices $\Pi \in \mathrm{Sym}(n)$.
\end{definition}

Thus convex graph invariants are convex functions that are \emph{constant over orbits} of the symmetric group acting on symmetric matrices by conjugation.  As described above the motivation behind the invariance property is clear. The motivation behind the convexity property is that we wish to construct solutions based on convex programming formulations in order to solve problems such as those listed in Section~\ref{sec:apps}.  We present several examples of convex graph invariants in Section~\ref{subsec:excgi}.  We note that a \emph{concave graph invariant} is a real-valued function over $\Sn$ that is the negative of a convex graph invariant.

We also consider invariant convex sets, which are defined in an analogous manner to convex graph invariants:
\begin{definition}
A set $C \subseteq \Sn$ is said to be an \emph{invariant convex set} if it is convex and if for any $A \in C$ it is the case that $\Pi A \Pi^T \in C$ for all permutation matrices $\Pi \in \mathrm{Sym}(n)$.
\end{definition}

In Section~\ref{subsec:exics} we present examples in which graphs can be constrained to have various properties by requiring that adjacency matrices belong to such convex invariant sets.  We also make connections between robust optimization and invariant convex sets in Section~\ref{subsec:rob}.

In order to systematically study convex graph invariants, we analyze certain elementary invariants that serve as a basis for constructing arbitrary convex invariants.  These elementary invariants are defined as follows:
\begin{definition}
An \emph{elementary convex graph invariant} is a function $\Theta_P: \Sn \rightarrow \R$ of the form
\begin{equation*}
\Theta_P(A) = \max_{\Pi \in \mathrm{Sym}(n)} \tr(P \Pi A \Pi^T),
\end{equation*}
for any $P \in \Sn$.
\end{definition}

It is clear that an elementary invariant is also a convex graph invariant, as it is expressed as the maximum over a set of convex functions (in fact linear functions).  We describe various properties of convex graph invariants in Sections~\ref{subsec:rep}.  One useful construction that we give is the expression of arbitrary convex graph invariants as suprema over elementary invariants.  We also discuss convex spectral invariants in Section~\ref{subsec:spec}, which are convex functions of a symmetric matrix that depend purely on its spectrum.  Finally an important point is that convex graph invariants may in general be hard to compute.  In Section~\ref{sec:comp} we discuss this problem and propose further tractable convex relaxations for cases in which a convex graph invariant may be intractable to compute.

In the Appendix we describe convex functions defined on $\R^n$ that are invariant with respect to any permutation of the argument.  Such functions have been analyzed previously, and we provide a list of their well-known properties.  We contrast these properties with those of convex graph invariants throughout the rest of this section.

\subsection{Examples of Convex Graph Invariants}
\label{subsec:excgi}

We list several examples of convex graph invariants.  As mentioned previously some of these invariants may possibly be difficult to compute, but we defer discussion of computational issues to Section~\ref{sec:comp}.  A useful property that we exploit in several of these examples is that a function defined as the supremum over a set of convex functions is itself convex \cite{Roc1996}.

\textbf{Number of edges.}  The total number of edges (or sum of edge weights) is an elementary convex graph invariant with $P = \tfrac{1}{2}(\ones \ones^T - I)$.

\textbf{Node weight.}  The maximum node weight of a graph, which corresponds to the maximum diagonal entry of the adjacency matrix of the graph, is an elementary convex graph invariant with $P = \e_1 \e_1^T$.  The maximum diagonal entry in \emph{magnitude} of an adjacency matrix is a convex graph invariant, and can be expressed as follows with $P = \e_1 \e_1^T$:
\begin{equation*}
\mathrm{max. ~ node ~ weight}(A) = \max\{\Theta_P(A), \Theta_{-P}(A)\}.
\end{equation*}
Similarly the sum of all the node weights, which is the sum of the diagonal entries of an adjacency matrix of a graph, can be expressed as an elementary convex graph invariant with $P$ being the identity matrix.

\textbf{Maximum degree.}  The maximum (weighted) degree of a node of a graph is also an elementary convex graph invariant with $P_{1,i} = P_{i,1} = 1, ~ \forall i \neq 1$, and all the other entries of $P$ set to zero.

\textbf{Largest cut.}  The value of the largest weighted cut of a graph specified by an adjacency matrix $A \in \Sn$ can be written as follows:
\begin{equation*}
\mathrm{max. ~ cut}(A) = \max_{\by \in \{-1,+1\}^n}~ \frac{1}{4}\sum_{i,j} A_{i,j} (1-\by_i \by_j).
\end{equation*}
As this function is a maximum over a set of linear functions, it is a convex function of $A$.  Further it is also clear that $\mathrm{max. ~cut}(A) = \mathrm{max. ~cut}(\Pi A \Pi^T)$ for all permutation matrices $\Pi$.  Consequently the value of the largest cut of a graph is a convex graph invariant.  We note here that computing this invariant is intractable in general.  In practice one could instead employ the following well-known tractable SDP relaxation \cite{GoeW1995}, which is related to the MAXCUT value by an appropriate shift and rescaling:
\begin{equation}
\begin{aligned}
f(A) = \min_{X \in \Sn} & ~~~ \tr(X A) \\ \mbox{s.t.} & ~~~ X_{ii} = 1, ~ \forall i \\ & ~~~ X \succeq 0.
\end{aligned}
\label{eq:maxcut1}
\end{equation}
As this relaxation is expressed as the minimum over a set of linear functions, it is a concave graph invariant.  In Section~\ref{subsec:gencomp} we discuss in greater detail tractable relaxations for invariants that are difficult to compute.

\textbf{Isoperimetric number (Cheeger constant).}  The isoperimetric number, also known as the Cheeger constant \cite{Dod1984}, of a graph specified by adjacency matrix $A \in \Sn$ is defined as follows:
\begin{equation*}
\mathrm{isoperimetric ~ number}(A) = \min_{U \subset \{1,\dots,n\}, |U| \leq \tfrac{n}{2}, \by \in \R^n, \by_U = 1, \by_{U^c} = -1} ~~ \sum_{i,j}~ \frac{A_{i,j} (1-\by_i \by_j)}{4 |U|}.
\end{equation*}
Here $U^c = \{1,\dots,n\} \backslash U$ denotes the complement of the set $U$, and $\by_U$ is the subset of the entries of the vector $\by$ indexed by $U$.  As with the last example, it is again clear that this function is a concave graph invariant as it is expressed as the minimum over a set of linear functions.  In particular it can be viewed as measuring the value of a ``normalized'' cut, and plays an important role in several aspects of graph theory \cite{HooLW2006}.

\textbf{Degree sequence invariants.}  Given a graph specified by adjacency matrix $A$ (assume for simplicity that the node weights are zero), the weighted \emph{degree sequence} is given by the vector $\bd(A) = \overline{A \ones}$, i.e., the vector obtained by sorting the entries of $A \ones$ in descending order.  It is easily seen that $\bd(A)$ is a graph invariant.  Consequently any function of $\bd(A)$ is also a graph invariant.  However our interest is in obtaining \emph{convex} functions of the adjacency matrix $A$.  An important class of functions of $\bd(A)$ that are convex functions of $A$, and therefore are convex graph invariants, are of the form:
\begin{equation*}
f(A) = \mathbf{v}^T \bd(A),
\end{equation*}
for $\mathbf{v} \in \R^n$ such that $\mathbf{v}_1 \geq \cdots \geq \mathbf{v}_n$.  This function can also be expressed as the maximum over all permutations $\Pi \in \mathrm{Sym}(n)$ of the inner-product $\mathbf{v}^T \Pi A \ones$.  As described in the Appendix such linear \emph{monotone} functionals can be used to express \emph{all} convex functions over $\R^n$ that are invariant with respect to permutations of the argument.  Consequently these monotone functions serve as building blocks for constructing all convex graph invariants that are functions of $\bd(A)$.

\textbf{Spectral invariants.}  Let the eigenvalues of the adjacency matrix $A$ of a graph be denoted as $\lambda_1(A) \geq \cdots \geq \lambda_n(A)$, and let $\lambda(A) = [\lambda_1(A), \dots, \lambda_n(A)]$.  These eigenvalues form the \emph{spectrum} of the graph specified by $A$, and clearly remain unchanged under transformations of the form $A \rightarrow V A V^T$ for any orthogonal matrix $V \in \mathrm{O}(n)$ (and therefore for any permutation matrix).  Hence any function of the spectrum of a graph is a graph invariant.  Analogous to the previous example, an important class of spectral functions that are also \emph{convex} are of the form:
\begin{equation*}
f(A) = \mathbf{v}^T \lambda(A),
\end{equation*}
for $\mathbf{v} \in \R^n$ such that $\mathbf{v}_1 \geq \cdots \geq \mathbf{v}_n$.  We denote spectral invariants that are also convex functions as \emph{convex spectral invariants}.  As with convex invariants of the degree sequence, all convex spectral invariants can be constructed using monotone functions of the type described here (see the Appendix).

\textbf{Second-smallest eigenvalue of Laplacian.}  This example is only meaningful for weighted graphs in which the node and edge weights are non-negative. For such a graph specified by adjacency matrix $A$, let $D_A = \mathrm{diag}(A \ones)$, where $\mathrm{diag}$ takes as input a vector and forms a diagonal matrix with the entries of the vector on the diagonal.  The \emph{Laplacian} of a graph is then defined as follows:
\begin{equation*}
L_A = D_A - A.
\end{equation*}
If $A \in \Sn$ consists of nonnegative entries, then $L_A \succeq 0$.  In this setting we denote the eigenvalues of $L_A$ as $\lambda_1(L_A) \geq \cdots \geq \lambda_n(L_A)$.  It is easily seen that $\lambda_n(L_A) = 0$ as the all-ones vector $\ones$ lies in the kernel of $L_A$.  The second-smallest eigenvalue $\lambda_{n-1}(L_A)$ of the Laplacian is a \emph{concave} invariant function of $A$.  It plays an important role as the graph specified by $A$ is connected if and only if $\lambda_{n-1}(L_A) > 0$.

\textbf{Inverse of Stability Number.}  A stable set of an unweighted graph $\mathcal{G}$ is a subset of the nodes of $\mathcal{G}$ such that no two nodes in the subset are adjacent.  The stability number is the size of the largest stable set of $\mathcal{G}$, and is denoted by $\alpha(\mathcal{G})$.  By a result of Motzkin and Straus \cite{MotS1965}, the inverse of the stability number can be written as follows:
\begin{equation}
\begin{aligned}
\frac{1}{\alpha(\mathcal{G})} = \min_{\bx} & ~~~ \bx^T (I+A) \bx \\ \mbox{s.t.} & ~~~ \bx_i \geq 0, ~ \forall i, ~~~ \sum_i \bx_i = 1.
\end{aligned}
\label{eq:mssn}
\end{equation}
Here $A$ is any adjacency matrix representing the graph $\mathcal{G}$.  Although this formulation is for unweighted graphs with edge weights being either one or zero, we note that the definition can in fact be \emph{extended} to all weighted graphs, i.e., for graphs with adjacency matrix given by \emph{any} $A \in \Sn$.  Consequently, the inverse of this extended stability number of a graph is a concave graph invariant over $\Sn$ as it is expressed as the minimum over a set of linear functions.  As this function is difficult to compute in general (because the stability number of a graph is intractable to compute), one could employ the following tractable relaxation:
\begin{equation}
\begin{aligned}
f(A) = \min_{X \in \Sn} & ~~~ \tr(X (I+A)) \\ \mbox{s.t.} & ~~~ X \geq 0, ~~~ X \succeq 0, ~~~ \ones^T X \ones = 1.
\end{aligned}
\label{eq:mssnrelax}
\end{equation}
This relaxation is also a concave graph invariant as it is expressed as the minimum over a set of affine functions.

\subsection{Examples of Invariant Convex Sets}
\label{subsec:exics}

Next we provide examples of invariant convex sets.  As described below constraints expressed using such sets are useful in order to require that graphs have certain properties.  Note that a sublevel set $\{A : f(A) \leq \alpha\}$ for any convex graph invariant $f$ is an invariant convex set.  Therefore all the examples of convex graph invariants given above can be used to construct invariant convex set constraints.

\textbf{Algebraic connectivity and diffusion.}  As mentioned in Section~\ref{subsec:excgi} a graph represented by adjacency matrix $A \in \A$ has the property that the second-smallest eigenvalue $\lambda_{n-1}(L_A)$ of the Laplacian of the graph is a concave graph invariant.  The constraint set $\{A : A \in \A, ~ \lambda_{n-1}(L_A) \geq \epsilon\}$ for any $\epsilon > 0$ expresses the property that a graph must be \emph{connected}.  Further if we set $\epsilon$ to be relatively large, we can require that a graph has good diffusion properties.

\textbf{Largest clique constraint.}  Let $K_k \in \Sn$ denote the adjacency matrix of an unweighted $k$-clique.  Note that $K_k$ is only nonzero within a $k \times k$ submatrix, and is zero-padded to lie in $\Sn$.  Consider the following invariant convex set for $\epsilon > 0$:
\begin{equation*}
\{ A : A \in \A, ~ \Theta_{K_k}(A) \leq (k^2 - k) - \epsilon\}.
\end{equation*}
This constraint set expresses the property that a graph cannot have a clique of size $k$ (or larger), with the edge weights of all edges in the clique being close to one.  For example we can use this constraint set to require that a graph has no triangles (with large edge weights).  It is important to note that triangles (and cliques more generally) are forbidden only with the qualification that all the edge weights in the triangle cannot be close to one.  For example a graph may contain a triangle with each edge having weight equal to $\tfrac{1}{2}$.  In this case the function $\Theta_{K_3}$ evaluates to $3$, which is much smaller than the maximum value of $6$ that $\Theta_{K_3}$ can take for matrices in $\A$ that contain a triangle with edge weights equal to one.

\textbf{Girth constraint.}  The girth of a graph is the length of the shortest cycle.  Let $C_k \in \Sn$ denote the adjacency matrix of an unweighted $k$-cycle for $k \leq n$.  As with the $k$-clique note that $C_k$ is nonzero only within a $k \times k$ submatrix, and is zero-padded so that it lies in $\Sn$.  In order to express the property that a graph has no small cycles, consider the following invariant convex set for $\epsilon > 0$:
\begin{equation*}
\{ A : A \in \A, ~ \Theta_{C_k}(A) \leq 2 k - \epsilon ~ \forall k \leq k_0\}.
\end{equation*}
Graphs belonging to this set cannot have cycles of length less than or equal to $k_0$, with the weights of edges in the cycle being close to one.  Thus we can impose a lower bound on a weighted version of the girth of a graph.

\textbf{Forbidden subgraph constraint.}  The previous two examples can be viewed as special cases of a more general constraint involving forbidden subgraphs.  Specifically let $A_k$ denote the adjacency matrix of an unweighted graph on $k$ nodes that consists of $E_k$ edges.  As before $A_k$ is zero-padded to ensure that it lies in $\Sn$.  Consider the following invariant convex set for $\epsilon > 0$:
\begin{equation*}
\{ A : A \in \A, ~ \Theta_{A_k}(A) \leq 2 E_k - \epsilon\}.
\end{equation*}
This constraint set requires that a graph not contain the subgraph given by the adjacency matrix $A_k$, with edge weights close to one.

\textbf{Degree distribution.}  Using the notation described previously, let $\bd(A) = \overline{A \ones}$ denote the sorted degree sequence ($\bd(A)_1 \geq \cdots \geq \bd(A)_n$) of a graph specified by adjacency matrix $A$.  We wish to consider the set of all graphs that have degree sequence $\bd(A)$.  This set is in general not convex unless $A$ represents a (weighted) regular graph, i.e., $\bd(A) = \alpha \ones$ for some constant $\alpha$.  Therefore we consider the \emph{convex hull} of all graphs that have degree sequence given by $\bd$:
\begin{equation*}
\mathcal{D}(A) = \mathrm{conv}\{B : B \in \Sn, ~ \overline{B \ones} = \bd(A)\}.
\end{equation*}
This set is in fact tractable to represent, and is given by the set of graphs whose degree sequence is \emph{majorized} by $\bd$:
\begin{equation*}
\mathcal{D}(A) = \left\{B : B \in \Sn, ~ \ones^T B \ones = \ones^T \bd(A), ~ \sum_{i=1}^k (\overline{B \ones})_i \leq \sum_{i=1}^k \bd(A)_i ~ \forall k = 1,\dots,n-1 \right\}.
\end{equation*}
By the majorization principle \cite{BenN2001} another representation for this convex set is as the set of graphs whose degree sequence lies in the \emph{permutahedron} generated by $\bd$ \cite{Zie1995}; the permutahedron generated by a vector is the convex hull of all permutations of the vector.  The notion of majorization is sometimes also referred to as \emph{Lorenz dominance} (see the Appendix for more details).

\textbf{Spectral distribution.}  Let $\lambda(A)$ denote the spectrum of a graph represented by adjacency matrix $A$.  As before we are interested in the set of all graphs that have spectrum $\lambda(A)$.  This set is nonconvex in general, unless $A$ is a multiple of the identity matrix in which case all the eigenvalues are the same.  Therefore we consider the convex hull of all graphs (i.e., symmetric adjacency matrices) that have spectrum equal to $\lambda(A)$:
\begin{equation*}
\mathcal{E}(A) = \mathrm{conv}\{B : B \in \Sn, ~ \lambda(B) = \lambda(A)\}.
\end{equation*}
This convex hull also has a tractable semidefinite representation analogous to the description above \cite{BenN2001}:
\begin{equation*}
\mathcal{E}(A)= \left\{B : B \in \Sn, ~ \tr(B) = \tr(A), ~ \sum_{i=1}^k \lambda(B)_i \leq \sum_{i=1}^k \lambda(A)_i ~ \forall k = 1,\dots,n-1 \right\}.
\end{equation*}
Note that eigenvalues are specified in descending order, so that $\sum_{i=1}^k \lambda(B)_i$ represents the sum of the $k$-largest eigenvalues of $B$.

\subsection{Representation of Convex Graph Invariants}
\label{subsec:rep}

All invariant convex sets and convex graph invariants can be represented using elementary convex graph invariants.  In this section we describe both these representation results.  Representation theorems in mathematics give expressions of complicated sets or functions in terms of simpler, basic objects.  In functional analysis the Riesz representation theorem relates elements in a Hilbert space and its dual, by uniquely associating each element of the Hilbert space to a linear functional \cite{Rud1966}.  In probability theory de Finetti's theorem states that a collection of exchangeable random variables can be expressed as a mixture of independent, identically distributed random variables.  In convex analysis every closed convex set can be expressed as the intersection of halfspaces \cite{Roc1996}. In each of these cases representation theorems provide a powerful analysis tool as they give a \emph{canonical} expression for complicated mathematical objects in terms of elementary sets/functions.

First we give a representation result for convex graph invariants.  In order to get a flavor of this result consider the maximum absolute-value node weight invariant of Section~\ref{subsec:excgi}, which is represented as the supremum over two elementary convex graph invariants.  The following theorem states that in fact any convex graph invariant can be expressed as a supremum over elementary invariants:

\begin{theorem} \label{theo:cgirep}
Let $f$ be any convex graph invariant.  Then $f$ can be expressed as follows:
\begin{equation*}
f(A) = \sup_{P \in \mathcal{P}} ~ \Theta_P(A) - \alpha_P,
\end{equation*}
for $\alpha_P \in \R$ and for some subset $\mathcal{P} \subset \Sn$.
\end{theorem}

\begin{proof}
Since $f$ is a convex function, it can be expressed as the supremum over linear functionals as follows:
\begin{equation*}
f(A) = \sup_{P \in \mathcal{P} \subseteq \Sn} ~ \tr(P A) - \alpha_P,
\end{equation*}
for $\alpha_P \in \R$.  This conclusion follows directly from the separation theorem in convex analysis \cite{Roc1996}; in particular this description of the convex function $f$ can be viewed as a specification in terms of supporting hyperplanes of the epigraph of $f$, which is a convex subset of $\Sn \times \R$.  However as $f$ is also a graph invariant, we have that $f(A) = f(\Pi A \Pi^T)$ for any permutation $\Pi$ and for all $A \in \Sn$.  Consequently for any permutation $\Pi$ and for any $P \in \mathcal{P}$,
\begin{equation*}
f(A) = f(\Pi A \Pi^T) \geq \tr(P \Pi A \Pi^T) - \alpha_P.
\end{equation*}
Thus we have that
\begin{equation}
f(A) \geq \sup_{P \in \mathcal{P}} ~ \Theta_P(A) - \alpha_P. \label{eq:cgirep1}
\end{equation}
However it also clear that for each $P \in \mathcal{P}$
\begin{equation*}
\Theta_P(A) - \alpha_P \geq \tr(P A) - \alpha_P,
\end{equation*}
which allows us to conclude that
\begin{equation}
\sup_{P \in \mathcal{P}} ~ \Theta_P(A) - \alpha_P \geq \sup_{P \in \mathcal{P}} ~ \tr(P A) - \alpha_P = f(A). \label{eq:cgirep2}
\end{equation}
Combining equations \eqref{eq:cgirep1} and \eqref{eq:cgirep2} we have the desired result.
\end{proof}

\begin{remark}
This result can be strengthened in the sense that one need only consider elements in $\mathcal{P}$ that lie in different equivalence classes up to conjugation by permutation matrices $\Pi \in \mathrm{Sym}(n)$.  In each equivalence class the representative functional is the one with the smallest value of $\alpha_P$.  This idea can be formalized as follows.  Consider the group action $\rho: (M,\Pi) \rightarrowtail \Pi M \Pi^T$ that conjugates elements in $\Sn$ by a permutation matrix in $\mathrm{Sym}(n)$.  With this notation we may restrict our attention in Theorem~\ref{theo:cgirep} to $\mathcal{P} \subset \Sn / \mathrm{Sym}(n)$, where $\Sn / \mathrm{Sym}(n)$ represents the quotient space under the group action $\rho$.  Such a mathematical object obtained by taking the quotient of a Euclidean space (or more generally a smooth manifold) under the action of a finite group is called an \emph{orbifold}.  With this strengthening one can show that there exists a \emph{unique, minimal} representation set $\mathcal{P} \subset \Sn / \mathrm{Sym}(n)$.  We however do not emphasize such refinements in subsequent results, and stick with the weaker statement that $\mathcal{P} \subseteq \Sn$ for notational and conceptual simplicity.
\end{remark}

As our next result we show that any invariant convex set can be represented as the intersection of sublevel sets of elementary convex graph invariants:

\begin{proposition} \label{theo:icsrep}
Let $\mathcal{S} \subseteq \Sn$ be an invariant convex set.  Then there exists a representation of $\mathcal{S}$ as follows:
\begin{equation*}
\mathcal{S} = \bigcap_{P \in \mathcal{P}} ~ \{A : A \in \Sn, ~ \Theta_P(A) \leq \alpha_P\},
\end{equation*}
for some $\mathcal{P} \subseteq \Sn$ and for $\alpha_P \in \R$.
\end{proposition}

\begin{proof}
The proof of this statement proceeds in an analogous manner to that of Theorem~\ref{theo:cgirep}, and is again essentially a consequence of the separation theorem in convex analysis.
\end{proof}

\subsection{A Robust Optimization View of Invariant Convex Sets}
\label{subsec:rob}

Uncertainty arises in many real-world problems.  An important goal in robust optimization (see \cite{BenEN2009} and the reference therein) is to translate formal notions of measures of uncertainty into convex constraint sets.  Convexity is important in order to obtain optimization formulations that are tractable.

The representation of a graph via an adjacency matrix in $\Sn$ is inherently uncertain as we have no information about the specific labeling of the nodes of the graph.  In this section we associate to each graph a convex polytope, which represents the best convex uncertainty set given a graph:

\begin{definition}
Let $\mathcal{G}$ be a graph that is represented by an adjacency matrix $A \in \Sn$ (any choice of representation is suitable).  The \emph{convex hull} of the graph $\mathcal{G}$ is defined as the following convex polytope:
\begin{equation*}
\mathcal{C}(\mathcal{G}) = \mathrm{conv}\{\Pi A \Pi^T : \Pi \in \mathrm{Sym}(n)\}.
\end{equation*}
\end{definition}

Recall that $\mathrm{Sym}(n)$ is the symmetric group of $n \times n$ permutation matrices.  One can check that the convex hull of a graph is an invariant convex set, and that its extreme points are the matrices $\Pi A \Pi^T$ for all $\Pi \in \mathrm{Sym}(n)$.  Note that this convex hull may in general be intractable to characterize; if these polytopes were tractable to characterize we would be able to solve the graph isomorphism problem in polynomial time.

The convex hull of a graph is the smallest convex set that contains all the adjacency matrices that represent the graph.  Therefore $\mathcal{C}(\mathcal{G})$ is in some sense the ``best convex characterization'' of the graph $\mathcal{G}$.  This notion is related to the concept of \emph{risk measures} studied in \cite{ArtDEH1999}, and the construction of convex uncertainty sets based on these risk measures studied in \cite{BerB2009}.  In particular we recall the following definition from \cite{BerB2009}:

\begin{definition}
Let $\mathcal{Z} = \{Z_1, \dots, Z_k\}$ be any finite collection of elements with $Z_i \in \Sn$.  Let $\bq \in \R^k$ be a probability distribution, i.e., $\sum_i \bq_i = 1$ and $\bq_i \geq 0, ~\forall i$.  Then the $\bq$-permutohull is the polytope in $\Sn$ defined as follows:
\begin{equation*}
\mathcal{B}_{\bq}(\mathcal{Z}) = \mathrm{conv}\left\{\sum_i (\Pi \bq)_i Z_i: \Pi \in \mathrm{Sym}(k) \right\}.
\end{equation*}
\end{definition}

Convex uncertainty sets given by permutohulls emphasize a data-driven view of robust optimization as adopted in \cite{BerB2009}.  Specifically the only information available about an uncertain set in many settings is a finite collection of data vectors $\mathcal{Z}$, and the probability distribution $\bq$ expresses preferences over such an unordered data set.  Therefore given a data set and a probability distribution that quantifies uncertainty with respect to elements of this data set, the $\bq$-permutohull is the smallest convex set expressing these uncertainty preferences.  We note that an important property of a permutohull is that it is invariant with respect to relabeling of the data vectors in $\mathcal{Z}$.

The convex hull of a graph $\mathcal{C}(\mathcal{G})$ is a simple example of a permutohull $\mathcal{B}_{\bq}(\mathcal{Z})$, with the distribution being $\bq = (1, 0, \dots, 0)$ and the set $\mathcal{Z} = \{\Pi A \Pi^T : \Pi \in \mathrm{Sym}(n)\}$ where $A \in \Sn$ represents the graph $\mathcal{G}$. More complicated permutohulls of graphs may be of interest in several applications but we do not pursue these generalizations here, and instead focus on the case of the convex hull of a graph as defined above.

The convex hull of a graph is itself an invariant convex set by definition.  Therefore we can appeal to Proposition~\ref{theo:icsrep} to give a representation of this set in terms of sublevel sets of elementary convex graph invariants.  As our next result we show that the values of all elementary convex graph invariants of $\mathcal{G}$ can be used to produce such a representation:

\begin{proposition} \label{theo:chrep}
Let $\mathcal{G}$ be a graph and let $A \in \Sn$ be an adjacency matrix representing $\mathcal{G}$.  We then have that
\begin{equation*}
\mathcal{C}(\mathcal{G}) = \bigcap_{P \in \Sn} ~ \{B : B \in \Sn, ~ \Theta_P(B) \leq \Theta_P(A) \}.
\end{equation*}
\end{proposition}

\begin{proof}
One direction of inclusion in this result is easily seen.  Indeed we have that for any $\Pi \in \mathrm{Sym}(n)$
\begin{equation*}
\Pi A \Pi^T \in \bigcap_{P \in \Sn} ~ \{B : B \in \Sn, ~ \Theta_P(B) \leq \Theta_P(A) \}.
\end{equation*}
As the right-hand-side is a convex set it is clear that the convex hull $\mathcal{C}(\mathcal{G})$ belongs to the set on the right-hand-side:
\begin{equation*}
\mathcal{C}(\mathcal{G}) \subseteq \bigcap_{P \in \Sn} ~ \{B : B \in \Sn, ~ \Theta_P(B) \leq \Theta_P(A) \}.
\end{equation*}

For the other direction suppose for the sake of a contradiction that we have a point $M \not\in \mathcal{C}(\mathcal{G})$ but with $\Theta_P(M) \leq \Theta_P(A)$ for all $P \in \Sn$.  As $M \not\in \mathcal{C}(\mathcal{G})$ we appeal to the separation theorem from convex analysis \cite{Roc1996} to produce a strict separating hyperplane between $M$ and $\mathcal{C}(\mathcal{G})$, i.e., a $\tilde{P} \in \Sn$ such that
\begin{equation*}
\tr(\tilde{P} B) < \alpha, \forall B \in \mathcal{C}(\mathcal{G}), ~~~ \mathrm{and} ~~~ \tr(\tilde{P} M) > \alpha.
\end{equation*}
Further as $\mathcal{C}(\mathcal{G})$ is an invariant convex set, it must be the case that
\begin{equation*}
\Theta_{\tilde{P}}(B) < \alpha, \forall B \in \mathcal{C}(\mathcal{G}).
\end{equation*}
On the other hand as $\tr(\tilde{P} M) > \alpha$ we also have that $\Theta_{\tilde{P}}(M) > \alpha$.  It is thus clear that
\begin{equation*}
\Theta_{\tilde{P}}(A) < \alpha < \Theta_{\tilde{P}}(M),
\end{equation*}
which leads us to a contradiction and concludes the proof.
\end{proof}

Therefore elementary convex graph invariants are useful for representing all the ``convex properties'' of a graph.  This result agrees with the intuition that the ``maximum amount of information'' that one can hope to obtain from convex graph invariants about a graph should be limited fundamentally by the convex hull of the graph.

As mentioned previously in many cases the convex hull of a graph may be intractable to characterize.  One can obtain outer bounds to this convex hull by using a tractable subset of elementary convex graph invariants; therefore we may obtain tractable but weaker convex uncertainty sets than the convex hull of a graph.  From Proposition~\ref{theo:chrep} such approximations can be refined as we use additional elementary convex graph invariants.  As an example the spectral convex constraint sets described in Section~\ref{subsec:exics} provide a tractable relaxation that plays a prominent role in our experiments in Section~\ref{sec:comp}.

\subsection{Comparison with Spectral Invariants}
\label{subsec:spec}

Convex functions that are invariant under certain group actions have been studied previously.  The most prominent among these is the set of convex functions of symmetric matrices that are invariant under conjugation by orthogonal matrices \cite{Dav1957}:
\begin{equation*}
f(M) = f(V M V^T), ~ \forall ~M \in \Sn, ~ \forall ~ V \in \mathrm{O}(n).
\end{equation*}
It is clear that such functions depend only on the spectrum of a symmetric matrix, and therefore we refer to them as \emph{convex spectral invariants}:
\begin{equation*}
f(M) = \tilde{f}(\lambda(M)),
\end{equation*}
where $\tilde{f}: \R^n \rightarrow \R$.  It is shown in \cite{Dav1957} that $f$ is convex if and only if $\tilde{f}$ is a convex function that is \emph{symmetric} in its argument:
\begin{equation*}
\tilde{f}(\bx) = \tilde{f}(\Pi \bx), ~ \forall \bx \in \R^n, \forall \Pi \in \mathrm{Sym}(n).
\end{equation*}
One can check that any convex spectral invariant can be represented as the supremum over monotone functionals of the spectrum of the form:
\begin{equation*}
\tilde{f}(\bx) = \mathbf{v}^T \overline{\bx} - \alpha,
\end{equation*}
for $\mathbf{v} \in \R^n$ such that $\mathbf{v}_1 \geq \cdots \geq \mathbf{v}_n$.  See the Appendix for more details.

A convex spectral invariant is also a convex graph invariant as invariance with respect to conjugation by any orthogonal matrix is a stronger requirement than invariance with respect to conjugation by any permutation matrix.  As many convex spectral invariants are tractable to compute, they form an important subclass of convex graph invariants.  In Section~\ref{subsec:qap} we discuss a natural approximation to elementary convex graph invariants using convex spectral invariants by replacing the symmetric group $\mathrm{Sym}(n)$ in the maximization by the orthogonal group $\mathrm{O}(n)$.  Finally one can define a spectrally invariant convex set $\mathcal{S}$ (analogous to invariant convex sets defined in Section~\ref{subsec:cgidef}) in which $M \in \mathcal{S}$ if and only if $V M V^T \in \mathcal{S}$ for all $V \in \mathrm{O}(n)$.  Such sets are very useful in order to impose various spectral constraints on graphs, and often have tractable semidefinite representations.

\subsection{Convex versus Non-Convex Invariants}
\label{subsec:nonci}

There are many graph invariants that are not convex.  In this section we give two examples that serve to illustrate the strengths and weaknesses of convex graph invariants.  First consider the spectral invariant given by the fifth largest eigenvalue of a graph, i.e., $\lambda_5(A)$ for a graph specified by adjacency matrix $A$.  This function is a graph invariant but it is not convex.  However from Section~\ref{subsec:excgi} we have that the \emph{sum} of the first five eigenvalues of a graph is a convex graph invariant.  More generally any function of the form $v_1 \lambda_1 + \cdots + v_5 \lambda_5$ with $v_1 \geq \cdots \geq v_5$ is a convex graph invariant.  Thus information about the fifth eigenvalue can be obtained in a ``convex manner'' only by including information about all the top five eigenvalues (or all the bottom $n-4$ eigenvalues).  As a second example consider the (weighted) sum of the total number of triangles that occur as subgraphs in a graph.  This function is again a non-convex graph invariant.  However recall from the forbidden subgraph example in Section~\ref{subsec:exics} that we can use elementary convex graph invariants to test whether a graph contains a triangle as a subgraph (with the edges of the triangle having large weights).  Therefore, roughly speaking convex graph invariants can be used to decide whether a graph contains a triangle, while general non-convex graph invariants can provide more information about the total number of triangles in a graph.  These examples demonstrate that convex graph invariants have certain limitations in terms of the type of information that they can convey about a graph.

The weaker form of information about a graph conveyed by convex graph invariants is nonetheless still useful in distinguishing between graphs.  As the next result demonstrates convex graph invariants are strong enough to distinguish between non-isomorphic graphs. This lemma follows from a straightforward application of Proposition~\ref{theo:chrep}:

\begin{lemma} \label{theo:complete}
Let $\mathcal{G}_1,\mathcal{G}_2$ be two non-isomorphic graphs represented by adjacency matrices $A_1,A_2 \in \Sn$, i.e., there exists no permutation $\Pi \in \mathrm{Sym}(n)$ such that $A_1 = \Pi A_2 \Pi^T$.  Then there exists a $P \in \Sn$ such that $\Theta_P(A_1) \neq \Theta_P(A_2)$.
\end{lemma}

\begin{proof}
Assume for the sake of a contradiction that $\Theta_P(A_1) = \Theta_P(A_2)$ for all $P \in \Sn$.  Then we have from Proposition~\ref{theo:chrep} that $\mathcal{C}(\mathcal{G}_1) = \mathcal{C}(\mathcal{G}_2)$.  As the extreme points of these polytopes must be the same, there must exist a permutation $\Pi \in \mathrm{Sym}(n)$ such that $A_1 = \Pi A_2 \Pi^T$.  This leads to a contradiction.
\end{proof}

Hence for any two given non-isomorphic graphs there exists an elementary convex graph invariant that evaluates to different values for these two graphs.  Consequently elementary convex graph invariants form a \emph{complete} set of graph invariants as they can distinguish between any two non-isomorphic graphs.


\section{Computing Convex Graph Invariants}
\label{sec:comp}

In this section we focus on efficiently computing and approximating convex graph invariants, and on tractable representations of invariant convex sets.  We begin by studying the question of computing elementary convex graph invariants, before moving on to more general convex invariants.

\subsection{Elementary Invariants and the Quadratic Assignment Problem}
\label{subsec:qap}

As all convex graph invariants can be represented using only elementary invariants, we initially focus on computing the latter.  Computing an elementary convex graph invariant $\Theta_P(A)$ for general $A,P$ is equivalent to solving the so-called Quadratic Assignment Problem (QAP) \cite{Cel1998}.  Solving QAP is hard in general, because it includes as a special case the Hamiltonian cycle problem; if $P$ is the adjacency matrix of the $n$-cycle, then for an unweighted graph specified by adjacency matrix $A$ we have that $\Theta_P(A)$ is equal to $2n$ if and only if the graph contains a Hamiltonian cycle.  However there are well-studied spectral and semidefinite relaxations for QAP, which we discuss next.

The \emph{spectral relaxation} of $\Theta_P(A)$ is obtained by replacing the symmetric group $\mathrm{Sym}(n)$ in the definition by the orthogonal group $\mathrm{O}(n)$:
\begin{equation}
\Lambda_P(A) = \max_{V \in \mathrm{O}(n)} ~ \tr(P V A V^T). \label{eq:spec1}
\end{equation}
Clearly $\Theta_P(A) \leq \Lambda_P(A)$ for all $A,P \in \Sn$.  As one might expect $\Lambda_P(A)$ has a simple closed-form solution \cite{FinBR1987}:
\begin{equation}
\Lambda_P(A) = \lambda(P)^T \lambda(A), \label{eq:spec2}
\end{equation}
where $\lambda(A),\lambda(P)$ are the eigenvalues of $A,P$ sorted in descending order.

The spectral relaxation offers a simple bound, but is quite weak in many instances.  Next we consider the well-studied \emph{semidefinite relaxation} for the QAP, which offers a tighter relaxation \cite{ZhaKRW1998}.  The main idea behind the semidefinite relaxation is that we can linearize $\Theta_P(A)$ as follows:
\begin{eqnarray*}
\Theta_P(A) &=& \max_{\Pi \in \mathrm{Sym}(n)} ~ \tr(P \Pi A \Pi^T) \\ &=& \max_{\bx \in \R^{n^2}, \bx = \mathrm{vec}(\Pi), \Pi \in \mathrm{Sym}(n)} ~ \langle \bx, (A \otimes P) \bx \rangle \\ &=& ~ \max_{\bx \in \R^{n^2}, \bx = \mathrm{vec}(\Pi), \Pi \in \mathrm{Sym}(n)} ~ \tr((A \otimes P) \bx \bx^T).
\end{eqnarray*}
Here $A \otimes P$ denotes the tensor product between $A$ and $P$, and $\mathrm{vec}$ denotes the operation that stacks the columns of a matrix into a single vector. Consequently it is of interest to characterize the following convex hull:
\begin{equation*}
\mathrm{conv}\{\bx \bx^T : \bx \in \R^{n^2}, ~\bx = \mathrm{vec}(\Pi), ~\Pi \in \mathrm{Sym}(n)\}.
\end{equation*}
There is no known tractable characterization of this set, and by considering tractable approximations the semidefinite relaxation to $\Theta_P(A)$ is then obtained as follows:
\begin{equation}
\begin{aligned}
\Omega_P(A) = \max_{\by \in \R^{n^2}, ~ Y \in \mathbf{S}(n^2)} & ~~~ \tr(P \otimes A) \\ \mbox{s.t.} & ~~~ \tr((I \otimes (J-I))Y + ((J-I) \otimes I)Y) = 0 \\ & ~~~ \tr(Y) - 2\by^T \ones = -n \\ & ~~~Y \geq 0, ~
 \begin{pmatrix}
  1 & \by^T \\
  \by & Y
 \end{pmatrix} \succeq 0.
\end{aligned}
\label{eq:sdpqap}
\end{equation}
We refer the reader to \cite{ZhaKRW1998} for the detailed steps involved in the construction of this relaxation.  This SDP relaxation gives an upper bound to $\Theta_P(A)$, i.e., $\Omega_P(A) \geq \Theta_P(A)$.  One can show that if the extra rank constraint
\begin{equation*}
\mathrm{rank}\begin{pmatrix}
  1 & \by^T \\
  \by & Y
 \end{pmatrix} = 1
\end{equation*}
is added to the SDP \eqref{eq:sdpqap}, then $\Omega_P(A) = \Theta_P(A)$. Therefore if the optimal value of the SDP \eqref{eq:sdpqap} is achieved at some $\hat{\by},\hat{Y}$ such that this rank-one constraint is satisfied, then the relaxation is tight, i.e., we would have that $\Omega_P(A) = \Theta_P(A)$.

While the semidefinite relaxation \eqref{eq:sdpqap} can in principle be computed in polynomial-time, the size of the variable $Y \in \mathbf{S}(n^2)$ means that even moderate size problem instances are not well-suited to solution by interior-point methods.  In many practical situations however, we often have that the matrix $P \in \Sn$ represents the adjacency matrix of some small graph on $k$ nodes with $k \ll n$, i.e., $P$ is nonzero only inside a $k \times k$ submatrix and is zero-padded elsewhere so that it lies in $\Sn$.  For example as discussed in Section~\ref{subsec:exics},  $P$ may represent the adjacency matrix of a triangle in a constraint expressing that a graph is triangle-free.  In such cases computing or approximating $\Theta_P(A)$ may be done more efficiently as follows:
\begin{enumerate}
\item{\bf Combinatorial enumeration.} For very small values of $k$ it is possible to compute $\Theta_P(A)$ efficiently even by explicit combinatorial enumeration.  The complexity of such a procedure scales as $\mathcal{O}(n^k)$.  This approach may be suitable if, for example, $P$ represents the adjacency matrix of a triangle.

\item{\bf Symmetry reduction.} For larger values of $k$, combinatorial enumeration may no longer be appropriate.  In these cases the special structure in $P$ can be exploited to reduce the size of the SDP relaxation \eqref{eq:sdpqap}.  Specifically, using the methods described in \cite{deKS2010} it is possible to reduce the size of the matrix variables from $\mathcal{O}(n^2) \times \mathcal{O}(n^2)$ to size $\mathcal{O}(kn) \times \mathcal{O}(kn)$.  More generally, it is also possible to exploit \emph{group symmetry} in $P$ to similarly reduce the size of the SDP \eqref{eq:sdpqap} (see \cite{deKS2010} for details).

\end{enumerate}

\subsection{Other Methods and Computational Issues}
\label{subsec:gencomp}

In many special cases in which computing convex graph invariants may be intractable, it is also possible to use other types of tractable semidefinite relaxations.  As described in Section~\ref{subsec:excgi} the MAXCUT value and the inverse stability number of graphs are invariants that are respectively convex and concave.  However both of these are intractable to compute, and as a result we must employ the SDP relaxations for these invariants as discussed in Section~\ref{subsec:excgi}.

Another issue that arises in practice is the \emph{representation} of invariant convex sets.  As an example, let $f(A)$ denote the SDP relaxation of the MAXCUT value as defined in \eqref{eq:maxcut1}.  As $f(A)$ is a concave graph invariant, we may be interested in representing convex constraint sets as follows:
\begin{equation*}
\{A : A \in \Sn, ~ f(A) \geq \alpha\} = \{A: A \in \Sn, ~ \tr(X A) \geq \alpha ~~ \forall X \in \Sn ~ \mathrm{s.t.} ~ X_{ii} = 1, ~ X \succeq 0\}.
\end{equation*}
In order to computationally represent such a set specified in terms of a universal quantifier, we appeal to convex duality.  Using the standard dual formulation of \eqref{eq:maxcut1}, we have that:
\begin{equation*}
\{A : A \in \Sn, ~ f(A) \geq \alpha\} = \{A: A \in \Sn, ~~ \exists Y ~ \mathrm{diagonal ~ s.t.} ~ A \succeq Y, ~ \tr(Y) \geq \alpha\}.
\end{equation*}
This reformulation provides a description in terms of existential quantifiers that is more suitable for practical representation.  Such reformulations using convex duality are well-known, and can be employed more generally (e.g., for invariant convex sets specified by sublevel sets of the inverse stability number or its relaxations in Section~\ref{subsec:excgi})
%


\section{Using Convex Graph Invariants in Applications}
\label{sec:appscgi}

In this section we give solutions to the stylized problems of Section~\ref{sec:apps} using convex graph invariants.  In order to properly state our results we begin with a few definitions.  All the convex programs in our numerical experiments are solved using a combination of the SDPT3 package \cite{TohTT} and the YALMIP parser \cite{Lof2004}.

\subsection{Preliminary Definitions}
\label{subsec:appsdef}

Let $C$ be a convex set in $\Sn$, and let $\bx \in C$ be any point in $C$. Following standard notions from convex analysis \cite{Roc1996}, the \emph{tangent cone} at $\bx$ with respect to $C$ is defined as follows:

\begin{definition}
Given a convex set $C$, the \emph{tangent cone} at a point $\bx \in C$ with respect to $C$ is the set of directions from $\bx$ to any other point in $C$:
\begin{equation*}
\T_C(\bx) = \{\alpha \bz : \bz = \by - \bx, \by \in C, \alpha \geq 0 \}.
\end{equation*}
\end{definition}

If $C$ is a convex set expressing a constraint in a convex program, the tangent cone at a point $\bx \in C$ can be viewed as the set of feasible directions at $\bx$ to other points in $C$.  Next we define the \emph{normal cone} at $\bx$ with respect to $C$, again following the usual conventions in convex analysis \cite{Roc1996}:

\begin{definition}
Given a convex set $C$, the \emph{normal cone} at a point $\bx \in C$ with respect to $C$ is the set of normal vectors to supporting hyperplanes of $C$ at $\bx$:
\begin{equation*}
\N_C(\bx) = \{\bz : \langle \bz,  \by - \bx \rangle \leq 0 ~ \forall \by \in C\}.
\end{equation*}
\end{definition}

The normal cone and the tangent cone are polars of each other \cite{Roc1996}. A key property of normal cones that we use in stating our results is that for any convex set $C \subseteq \Sn$, the normal cones at all the extreme points of $C$ form a \emph{partition}\footnote{Note that there may be overlap on the boundaries of the normal cones at the extreme points, but these overlaps have smaller dimension than those of the normal cones.} of $\Sn$ \cite{Roc1996}.


\subsection{Application: Graph Deconvolution}
\label{subsec:dec}

Given a combination of two graphs overlaid on the same set of nodes, the graph deconvolution problem is to recover the individual graphs (as introduced in Section~\ref{subsec:appsdec}).

\begin{problem}
Let $\mathcal{G}_1$ and $\mathcal{G}_2$ be two graphs specified by particular adjacency matrices $A^\ast_1, A^\ast_2 \in \Sn$.  We are given the sum $A = A^\ast_1 + A^\ast_2$, and the additional information that $A^\ast_1,A^\ast_2$ correspond to particular realizations (labelings of nodes) of $\mathcal{G}_1, \mathcal{G}_2$.  The goal is to recover $A^\ast_1$ and $A^\ast_2$ from $A$.
\end{problem}

See Figure~\ref{fig:dec1} for an example illustrating this problem.  The key unknown in this problem is the specific labeling of the nodes of $\mathcal{G}_1$ and $\mathcal{G}_2$ relative to each other in the composite graph represented by $A$.  As described in Section~\ref{subsec:rob}, the best convex constraints that express this uncertainty are the convex hulls of the graphs $\mathcal{G}_1, \mathcal{G}_2$.  Therefore we consider the following natural solution based on convex optimization to solve the deconvolution problem:

\begin{solution}
Recall that $\mathcal{C}(\mathcal{G}_1)$ and $\mathcal{C}(\mathcal{G}_2)$ are the convex hulls of the unlabeled graphs $\mathcal{G}_1,\mathcal{G}_2$ (which we are given), and that $\|\cdot\|$ denotes the Euclidean norm.  We propose the following convex program to recover $A_1, A_2$:
\begin{equation}
\begin{aligned}
(\hat{A_1},\hat{A_2}) = \arg \min_{A_1,A_2 \in \Sn} & ~~~ \|A - A_1 - A_2\| \\ \mbox{s.t.} & ~~~ A_1 \in \mathcal{C}(\mathcal{G}_1), ~ A_2 \in \mathcal{C}(\mathcal{G}_2).
\end{aligned}
\label{eq:deconv1}
\end{equation}
One could also use in the objective any other norm that is invariant under conjugation by permutation matrices.  This program is convex, although it may not be tractable if the sets $\mathcal{C}(\mathcal{G}_1),\mathcal{C}(\mathcal{G}_2)$ cannot be efficiently represented.  Therefore it may be desirable to use tractable convex relaxations $C_1,C_2$ of the sets $\mathcal{C}(\mathcal{G}_1), \mathcal{C}(\mathcal{G}_2)$, i.e., $\mathcal{C}(\mathcal{G}_1) \subseteq C_1 \subset \Sn$ and $\mathcal{C}(\mathcal{G}_2) \subseteq C_2 \subset \Sn$:
\begin{equation}
\begin{aligned}
(\hat{A_1},\hat{A_2}) = \arg \min_{A_1,A_2 \in \Sn} & ~~~ \|A - A_1 - A_2\| \\ \mbox{s.t.} & ~~~ A_1 \in C_1, ~ A_2 \in C_2.
\end{aligned}
\label{eq:deconv2}
\end{equation}
\end{solution}

Recall from Proposition~\ref{theo:chrep} that we can represent $\mathcal{C}(\mathcal{G})$ using all the elementary convex graph invariants.  Tractable relaxations to this convex hull may be obtained, for example, by just using spectral invariants, degree-sequence invariants, or any other subset of invariant convex set constraints that can be expressed efficiently.  We give numerical examples later in this section.  The following result gives conditions under which we can exactly recover $A^\ast_1,A^\ast_2$ using the convex program \eqref{eq:deconv2}:

\begin{proposition} \label{theo:dec}
Given the problem setup as described above, we have that $(\hat{A_1},\hat{A_2}) = (A^\ast_1,A^\ast_2)$ is the \emph{unique optimum} of \eqref{eq:deconv2} if and only if:
\begin{equation*}
T_{C_1}(A^\ast_1) \cap -T_{C_2}(A^\ast_2) = \{0\},
\end{equation*}
where $-T_{C_2}(A^\ast_2)$ denotes the negative of the tangent cone $T_{C_2}(A^\ast_2)$.
\end{proposition}

\begin{proof}
Note that in the setup described above $(A^\ast_1,A^\ast_2)$ is an optimal solution of the convex program \eqref{eq:deconv2} as this point is feasible (since by construction $A^\ast_1 \in \mathcal{C}(\mathcal{G}_1) \subseteq C_1$ and $A^\ast_2 \in \mathcal{C}(\mathcal{G}_2) \subseteq C_2$), and the cost function achieves its minimum at this point.  This result is concerned with $(A^\ast_1,A^\ast_2)$ being the \emph{unique} optimal solution.

For one direction suppose that $T_{C_1}(A^\ast_1) \cap -T_{C_2}(A^\ast_2) = \{0\}$.  Then there exists no $Z_1 \in T_{C_1}(A^\ast_1), Z_2 \in T_{C_2}(A^\ast_2)$ such that $Z_1 + Z_2 = 0$ with $Z_1 \neq 0, Z_2 \neq 0$.  Consequently every feasible direction from $(A^\ast_1,A^\ast_2)$ into $C_1 \times C_2$ would increase the value of the objective.  Thus $(A^\ast_1,A^\ast_2)$ is the unique optimum of \eqref{eq:deconv2}.

For the other direction suppose that $(A^\ast_1, A^\ast_2)$ is the unique optimum of \eqref{eq:deconv2}, and assume for the sake of a contradiction that $T_{C_1}(A^\ast_1) \cap -T_{C_2}(A^\ast_2)$ contains a nonzero element, which we'll denote by $Z$.  There exists a scalar $\alpha > 0$ such that $A^\ast_1 + \alpha Z \in C_1$ and $A^\ast_2 - \alpha Z \in C_2$.  Consequently $(A^\ast_1 + \alpha Z, A^\ast_2 - \alpha Z)$ is also a feasible solution that achieves the lowest possible cost of zero.  This contradicts the assumption that $(A^\ast_1, A^\ast_2)$ is the unique optimum.
\end{proof}

Thus we have that \emph{transverse intersection}  of the tangent cones $T_{C_1}(A^\ast_1)$ and $-T_{C_2}(A^\ast_2)$ is equivalent to \emph{exact recovery} of $(A^\ast_1,A^\ast_2)$ given the sum $A = A^\ast_1+A^\ast_2$.  As $\mathcal{C}(\mathcal{G}_1) \subseteq C_1$ and $\mathcal{C}(\mathcal{G}_2) \subseteq C_2$, we have that $T_{\mathcal{C}(\mathcal{G}_1)}(A^\ast_1) \subseteq T_{C_1}(A^\ast_1)$ and $T_{\mathcal{C}(\mathcal{G}_2)} \subseteq T_{C_2}(A^\ast_2)$.  These relations follow from the fact that the set of feasible directions from $A^\ast_1$ and $A^\ast_2$ into the respective convex sets is enlarged.  Therefore the tangent cone transversality condition of Proposition~\ref{theo:dec} is generally more difficult to satisfy if we use relaxations $C_1,C_2$ to the convex hulls $\mathcal{C}(\mathcal{G}_1), \mathcal{C}(\mathcal{G}_2)$.  Consequently we have a \emph{tradeoff} between the complexity of solving the convex program, and the possibility of exactly recovering $(A^\ast_1,A^\ast_2)$.  However the following example suggests that it is possible to obtain tractable relaxations that still allow for perfect recovery.

\begin{figure}
\begin{center}
\epsfig{file=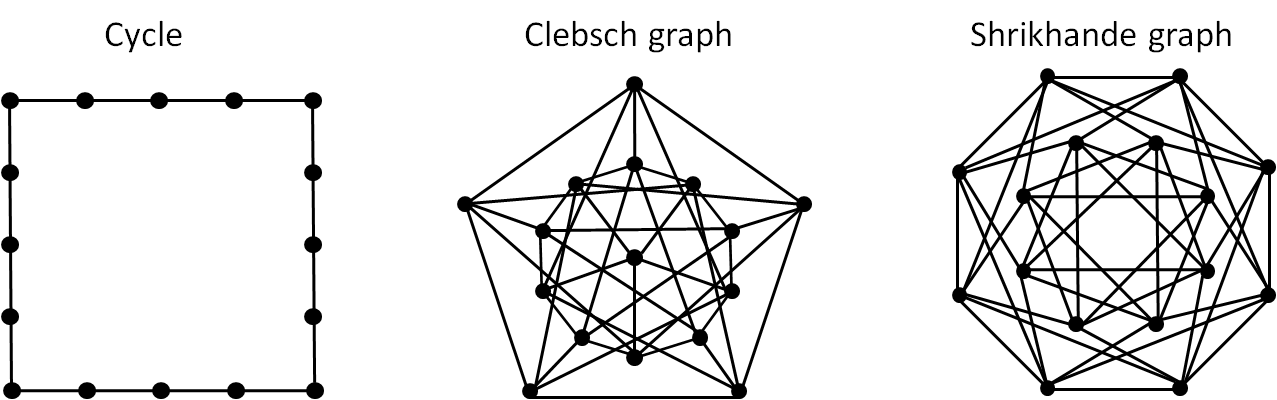,width=12cm,height=4cm} \caption{The three graphs used in the deconvolution experiments of Section~\ref{subsec:dec}.  The Clebsch graph and the Shrikhande graph are examples of strongly regular graphs on $16$ nodes \cite{GodR2004}; see Section~\ref{subsec:dec} for more details about the properties of such graphs.} \label{fig:dec2}
\end{center}
\end{figure}

\textbf{Example.} We consider the $16$-cycle, the Shrikhande graph, and the Clebsch graph (see Figure~\ref{fig:dec2}), and investigate the deconvolution problem for all three pairings of these graphs.  For illustration purposes suppose $A^\ast_1$ is an adjacency matrix of the unweighted $16$-node cycle denoted $\mathcal{G}_1$, and that $A^\ast_2$ is an adjacency matrix of the $16$-node Clebsch graph denoted $\mathcal{G}_2$ (see Figure~\ref{fig:dec1}).  These adjacency matrices are random instances chosen from the set of all valid adjacency matrices that represent the graphs $\mathcal{G}_1,\mathcal{G}_2$.  Given the sum $A = A^\ast_1 + A^\ast_2$, we construct convex constraint sets $C_1,C_2$ as follows:
\begin{eqnarray*}
C_1 &=& \A ~ \cap ~ \mathcal{E}(A^\ast_1) \\ C_2 &=& \A ~ \cap ~ \mathcal{E}(A^\ast_2).
\end{eqnarray*}
Here $\mathcal{E}(A)$ represents the spectral constraints of Section~\ref{subsec:exics}.  Therefore the graphs $\mathcal{G}_1$ and $\mathcal{G}_2$ are characterized purely by their spectral properties.  By running the convex program described above for $100$ random choices of labelings of the vertices of the graphs $\mathcal{G}_1,\mathcal{G}_2$, we obtained \emph{exact} recovery of the adjacency matrices $(A^\ast_1,A^\ast_2)$ in all cases (see Table~\ref{tab:sum}).  \emph{Thus we have exact decomposition based only on convex spectral constraints, in which the only invariant information used to characterize the component graphs $\mathcal{G}_1, \mathcal{G}_2$ are the spectra of $\mathcal{G}_1,\mathcal{G}_2$}.  Similarly successful decomposition results using only spectral invariants are also seen in the cycle/Shrikhande graph deconvolution problem, and the Clebsch graph/Shrikhande graph deconvolution problem; Table~\ref{tab:sum} gives complete results.

The inspiration for using the Clebsch graph and the Shrikhande graph as examples for deconvolution is based on Proposition~\ref{theo:dec}. Specifically, a graph for which the tangent cone with respect to the corresponding spectral constraint set $\mathcal{E}(A)$ (defined in Section~\ref{subsec:exics}) is small is well-suited to being deconvolved from other graphs using spectral invariants.  This is because the tangent cone being smaller implies that the transversality condition of Proposition~\ref{theo:dec} is easier to satisfy.  In order to obtain small tangent cones with respect to spectral constraint sets, we seek graphs that have many \emph{repeated eigenvalues}.  \emph{Strongly regular graphs}, such as the Clebsch graph and the Shrikhande graph, are prominent examples of graphs with repeated eigenvalues as they have only three distinct eigenvalues.  A strongly regular graph is an unweighted regular graph (i.e., each node has the same degree) in which every pair of adjacent vertices have the same number of common neighbors, and every pair of non-adjacent vertices have the same number of common neighbors \cite{GodR2004}.  We explore in more detail the properties of these and other graph classes in a separate report \cite{ChaPW1}, where we characterize families of graphs for which the transverse intersection condition of Proposition~\ref{theo:dec} provably holds for constraint sets $C_1,C_2$ constructed using tractable graph invariants.

\begin{table}[t]
\centering
\begin{tabular}{||c|c||}\hline
 Underlying graphs & $\#$ successes in $100$ random trials \\ \hline\hline
 The $16$-cycle and the Clebsch graph & 100 \\ \hline
 The $16$-cycle and the Shrikhande graph & 96 \\ \hline
 The Clebsch graph and the Shrikhande graph & 94\\
 \hline
\end{tabular}
\caption{A summary of the results of graph deconvolution via convex optimization:  We generated $100$ random instances of each deconvolution problem by randomizing over the labelings of the components.  The convex program uses only spectral invariants to characterize the convex hulls of the component graphs, as described in Section~\ref{subsec:dec}.} \label{tab:sum}
\end{table}


\subsection{Application: Generating Graphs with Desired Properties}
\label{subsec:id}

We first consider the problem of constructing a graph with certain desired structural properties.

\begin{problem}
Suppose we are given structural constraints on a graph in terms of a collection of (possibly nonconvex) graph invariants $\{h_j(A) = \alpha_j\}$.  Can we recover a graph that is consistent with these constraints?  For example we may be given constraints on the spectrum, the degree distribution, the girth, and the MAXCUT value.  Can we construct some graph $\mathcal{G}$ that is consistent with this knowledge?
\end{problem}

This problem may be infeasible in that there may no graph consistent with the given information.  We do not address this feasibility question here, and instead focus only on the computational problem of generating graphs that satisfy the given constraints assuming such graphs do exist.  Next we propose a convex programming approach using invariant convex sets to construct a graph $\mathcal{G}$, specified by an adjacency matrix $A$, which satisfies the required constraints.  Both the problem as well the solution can be suitably modified to include inequality constraints.

\begin{solution}
We combine information from all the invariants to construct an invariant convex set $C$.  Given a constraint of the form $h_j(A) = \alpha_j$, we consider the following convex set:
\begin{equation*}
C_j = \mathrm{conv}\{A : A \in \Sn, ~ h_j(A) = \alpha_j\}.
\end{equation*}
This set is convex by construction, and is an invariant convex set if $h_j$ is a graph invariant.  If $h_j$ is a convex graph invariant this set is equal to the sublevel set $\{A : A \in \Sn, ~ h_j(A) \leq \alpha_j\}$.  Given a collection of constraints $\{h_j(A) = \alpha_j\}$ we then form an invariant convex constraint set as follows:
\begin{equation*}
C = \cap_j ~~ C_j.
\end{equation*}
Therefore any invariant information that is amenable to approximation as a convex constraint set can be incorporated in such a framework.  For example constraints on the degree distribution or the spectrum can be naturally relaxed to tractable convex constraints, as described in Section~\ref{subsec:exics}.  If the set $C$ as defined above is intractable to compute, one may further relax $C$ to obtain efficient approximations.  In many cases of interest a subset of the boundary of $C$ corresponds to points at which \emph{all} the constraints are active $\{A : h_j(A) = \alpha_j\}$.  In order to recover one of these extreme points, we maximize a \emph{random} linear functional defined by $M \in \Sn$ (with the entries in the upper-triangular part chosen to be independent and identically distributed to zero-mean, variance-one standard Gaussians) over the set $C$:
\begin{equation}
\begin{aligned}
\hat{A} = \arg\max_{A \in \Sn} & ~~~ \tr(M A) \\ \mbox{s.t.} & ~~~ A \in C.
\end{aligned}
\label{eq:iden1}
\end{equation}
This convex program is successful if $\hat{A}$ is indeed an extreme point at which all the constraints $\{h_j(A) = \alpha_j\}$ are satisfied.
\end{solution}

Clearly this approach is well-suited for constructing constrained graphs only if the convex set $C$ described in the solution scheme contains many extreme points at which all the constraints are satisfied.  The next result gives conditions under which the convex program recovers an $\hat{A}$ that satisfies all the given constraints:

\begin{proposition} \label{theo:ht}
Consider the problem and solution setup as defined above.  Define the set $\N$ as follows:
\begin{equation*}
\N = \bigcup_{\{A ~:~ A \in C, ~ h_j(A) = \alpha_j ~\forall j\}} ~~~ \N_C(A).
\end{equation*}
If $M \in \N$ then the optimum $\hat{A}$ of the convex program \eqref{eq:iden1} satisfies all the specified constraints exactly.  In particular if $M$ is chosen uniformly at random as described above, then the probability of success is equal to the fraction of $\Sn$ covered by the union of the normal cones $\N$.
\end{proposition}

\begin{proof}
The proof follows from standard results in convex analysis.  In particular we appeal to the fact that a linear functional defined by $M$ achieves its maximum at $\hat{A} \in C$ if and only if $M \in \N_C(\hat{A})$.
\end{proof}

\begin{figure}
\begin{center}
\epsfig{file=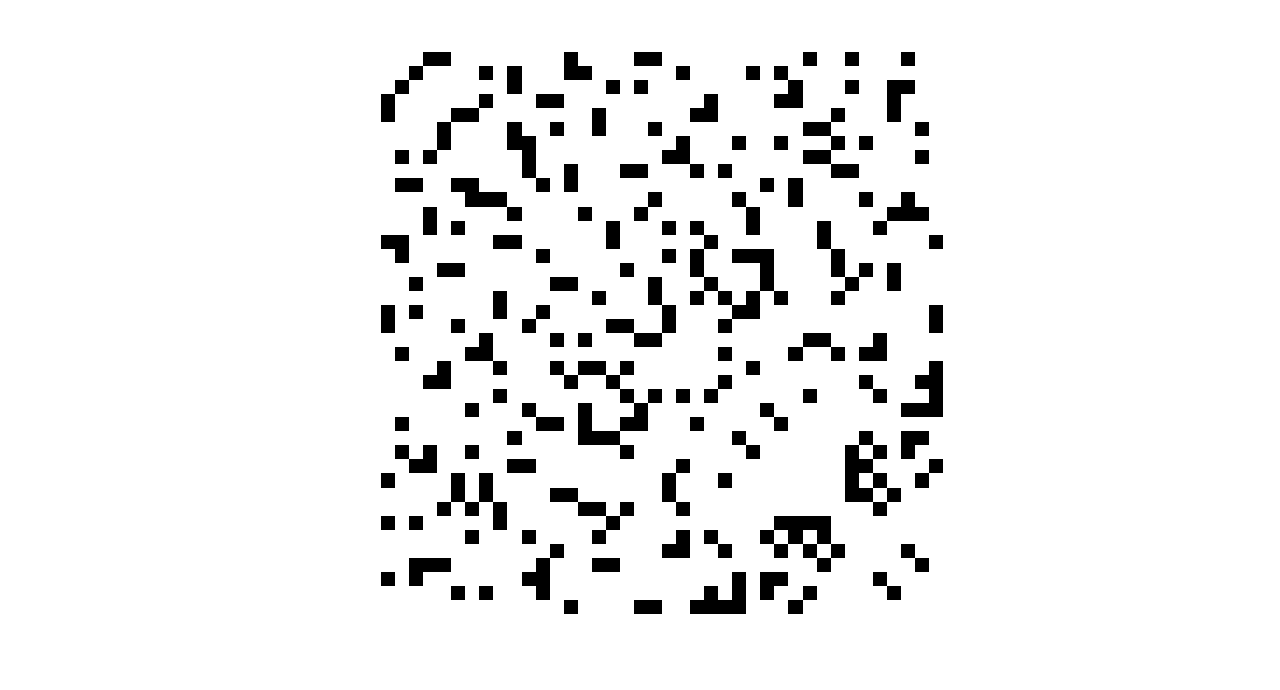,width=10cm,height=5cm} \caption{An adjacency matrix of a sparse, well-connected graph example obtained using the approach described in Section~\ref{subsec:id}: The weights of this graph lie in the range $[0,1]$, the black points represent edges with nonzero weight, and the white points denote absence of edges.  The (weighted) degree of each node is $8$, the average number of nonzero (weighted) edges per node is $8.4$, the second-smallest eigenvalue of the Laplacian is $4$, and the weighted diameter is $3$.} \label{fig:expander}
\end{center}
\end{figure}

As a corollary of this result we observe that if the invariant information provided exactly characterizes the convex hull of a graph $\mathcal{G}$, then the set $C$ above is the convex hull $\mathcal{C}(\mathcal{G})$ of the graph $\mathcal{G}$.  In such cases the convex program given by \eqref{eq:iden1} produces an adjacency matrix representing $\mathcal{G}$ with probability one.
Next we provide the results of a simple experiment that demonstrates the effectiveness of our approach in generating sparse graphs with large spectral gap.

\textbf{Example.} In this example we aim to construct graphs on $n = 40$ nodes with adjacency matrices in $\A$ that have degree $d = 8$, node weights equal to zero, and the second-smallest eigenvalue of the Laplacian being larger than $\epsilon = 4$.  The goal is to produce relatively \emph{sparse} graphs that satisfy these constraints.  The specified constraints can be used to construct a convex set as follows:
\begin{equation*}
C = \{A: A \in \A, ~ \tfrac{1}{8} A \ones = \ones, ~ \lambda_{n-1}(L_A) \geq 4, ~ A_{ii} = 0 ~ \forall i\}.
\end{equation*}
By maximizing $100$ random linear functionals over this set we obtained graphs in all $100$ cases with total degree equal to $8$, and in $98$ of the $100$ cases with the minimum eigenvalue of the Laplacian equal to $4$ (it is greater than $4$ in the remaining two cases).  Interestingly the average number of edges with nonzero weight incident on each node is $8.8$ over these $100$ trials, thus providing very sparse graphs that are well-connected.  Figure~\ref{fig:expander} gives an example of a graph generated randomly using this procedure; the average number of nonzero (weighted) edges per node of this graph is $8.4$, and its (weighted) diameter is $3$.  Therefore this approach empirically yields sparse graphs that are well-connected (i.e., with a large spectral gap).

We would like to point out here a different approach to constructing well-connected graphs, which tries to add edges from a subset of candidate edges to maximize the second eigenvalue of the graph Laplacian \cite{GhoB2006}.  An interesting question is to understand the structure of the extreme points of the set $C$ in this example as the graph size and the degree $(n,d)$ grow large, with $\epsilon$ held constant.  For example it may be useful to compute the fraction of the normal cones at those extreme points corresponding to expander graphs.  More generally it is of interest to give conditions on constraint sets under which the procedure described above is successful in providing graphs that satisfy all the constraints with high probability.


\subsection{Application: Graph Hypothesis Testing}
\label{subsec:ht}

Finally we give a solution to the hypothesis testing problem in which we have two families of graphs, and the goal is to decide which of these families offers a ``better explanation'' for a given candidate ``sample'' graph.

\begin{problem}
Let $\mathcal{F}_1$ and $\mathcal{F}_2$ denote two families of graphs characterized in terms of invariants $\{h_j^1\}$ and $\{h_j^2\}$ respectively; for example, a family could be specified as some set of graphs that have similar spectral distributions, similar degree sequences, and similar girths.  Given a graph $\mathcal{G}$, which of the two families $\mathcal{F}_1, \mathcal{F}_2$ of graphs is more similar to $\mathcal{G}$?
\end{problem}

We emphasize that the sets of invariants that characterize $\mathcal{F}_1,\mathcal{F}_2$ may in general be very different.  Note that this question is not completely well-posed, as there may be different answers depending on one's notion of similarity.  In order to address this point, we need to develop a statistical theory for graphs.  In such a setting one could phrase this question formally as a statistical hypothesis testing problem with appropriate error metrics.  Our focus in the present paper is on proposing a convex optimization solution to the hypothesis testing based on convex graph invariants, and using a reasonable notion of similarity.

\begin{solution}
Let $A \in \Sn$ be an adjacency matrix that represents the graph $\mathcal{G}$.  We construct invariant convex sets $C_1$ and $C_2$ based on the sets of invariants $\{h_j^1\}, \{h_j^2\}$ in an analogous manner to the construction described in the solution to the graph construction problem of Section~\ref{subsec:id}.  As before one could employ further tractable relaxations of these sets if they are intractable to compute.  Assuming that these convex constraint sets that summarize the families $\mathcal{F}_1$ and $\mathcal{F}_2$ are compact, we declare that $\mathcal{F}_1$ is closer to $\mathcal{G}$ than $\mathcal{F}_2$ if the following holds:
\begin{equation}
\max_{M \in C_1} ~ \tr(A M) ~~~ \geq ~~~ \max_{M \in C_2} ~ \tr(A M). \label{eq:ht}
\end{equation}
Naturally we declare the opposite result if the inequality is switched.  Computing the two sides in this test can be done via convex optimization, and this computation is tractable if $C_1,C_2$ are tractable to characterize.  \end{solution}

Our choice of the function to be maximized over $C_1,C_2$ is motivated by a similar procedure in statistics and signal processing, which goes by the name of ``matched filtering.''  Of course other (convex invariant) cost functions can also be optimized depending on one's notion of similarity.  We point out two advantages of this approach to hypothesis testing.  First the two families of graphs can be specified in terms of different sets of invariants, as seen in these examples.  Second the optimal solutions of the convex programs in \eqref{eq:ht} in fact provide \emph{approximations} to the graph $\mathcal{G}$ by elements in the families $\mathcal{F}_1,\mathcal{F}_2$.  We give illustrations of these points in our examples, which we describe next.

\textbf{Example.}  Let $A_{\mathrm{cycle}}$ denote the adjacency matrix of a $16$-node unweighted cycle.  As our first family we consider the set of cycles on $16$ nodes.  We approximate this family by the set of graphs that are triangle-free (in the sense described in Section~\ref{subsec:exics}), have degree equal to $2$, and have the same spectrum as a $16$-node unweighted cycle.  Therefore the set $C_1$ is defined as follows:
\begin{equation*}
C_1 = \{A : A \in \A, ~ A_{ii} = 0 ~ \forall i, ~ \tfrac{1}{2} A \ones = \ones, ~ \Theta_{K_3}(A) \leq 4\} \cap \mathcal{E}(A_{\mathrm{cycle}}).
\end{equation*}
As our second family, we consider sparse well-connected graphs on $16$ nodes with maximum weighted degree less than or equal to $2.5$, and with the second-smallest eigenvalue of the Laplacian bounded below by $1.1$:
\begin{equation*}
C_2 = \{A: A \in \A, ~ A_{ii} = 0 ~ \forall i, ~ (A \ones)_i \leq 2.5 ~ \forall i, ~ \lambda_{n-1}(L_A) \geq 1.1\}.
\end{equation*}
Applying the solution described above to a test graph given by a $16$-node unweighted path graph (i.e., an unweighted cycle with an edge removed, see Figure~\ref{fig:ht}), we find that the path graph is ``closer'' to the family $\mathcal{F}_1$ of cycles approximated by the set $C_1$ than it is to the family $\mathcal{F}_2$.  This agrees with the intuition that a path graph is not well-connected, and is only one edge away from being a cycle.  We also point out that the optimal solution to the convex program on the left-hand-side of the test \eqref{eq:ht} is in fact an unweighted $16$-node cycle with the missing edge in the path graph added as an extra edge.  Next we consider a different test graph -- a $16$-node cycle with two additional edges across diametrically opposite nodes, i.e., assuming we label the nodes of the $16$-node cycle we add edges between nodes $1$ and $9$, and between nodes $5$ and $13$ (again see Figure~\ref{fig:ht}).  While this graph is only two edges away from being a cycle, the edges connecting far away nodes dramatically increase the connectivity of the graph.  In this case we find using the convex programming hypothesis test \eqref{eq:ht} that the family $\mathcal{F}_2$ is in fact closer than $\mathcal{F}_1$ to the sample graph.  Interestingly, the optimal solution to the convex program on the left-hand-side of the test \eqref{eq:ht} is again an unweighted $16$-node cycle, this time with the two additional edges removed.

In order to thoroughly address the graph hypothesis testing problem, we need to develop a framework of statistical models over spaces of graphs.  With a proper statistical framework in place we can evaluate the \emph{probability of error} achieved by a hypothesis-testing algorithm with respect to a suitable error-metric, analogous to similar methods developed in other classical decision-theoretic problems in statistics.  We defer these questions to a separate paper.


\section{Discussion}
\label{sec:conc}

In this paper we introduced and studied convex graph invariants, which are graph invariants that are convex functions of the adjacency matrix.  Convex invariants form a rich subset of the set of all graph invariants, and they are useful in developing a unified computational framework based on convex optimization to solve a number of graph problems.  In particular we described three canonical problems involving the structural properties of graphs, namely, graph construction given constraints, graph deconvolution of a composite graph into individual components, and graph hypothesis testing in which the objective is decide which of two given families of graphs offers a better explanation for a new sample graph.  We presented convex optimization solutions to all of these problems, with convex graph invariants playing a prominent role.  These solutions provided attractive empirical performance, and the resulting convex programs are tractable and can be solved using general-purpose off-the-shelf software for moderate size instances.

We are presently investigating several research questions arising from this paper.  It is of interest to provide theoretical guarantees on the performance of our convex programs in Section~\ref{sec:appscgi} in solving the problems of Section~\ref{sec:apps}.  For example which families of graphs can be deconvolved or efficiently sampled from using convex optimization?  It is also preferable to develop special-purpose software to efficiently compute some subset of convex graph invariants, in order to enable the solution of very large problem instances.  Finally in order to properly analyze the success of algorithms for graph deconvolution, sampling, and hypothesis testing, it is important to develop a formal statistical framework for graphs.



\appendix

\section{Properties of Convex Symmetric Functions}\label{app:csf}

A \emph{convex symmetric function} is a convex function that is invariant with respect to a permutation of the argument:
\begin{definition}
A function $g: \R^n \rightarrow \R$ is a \emph{convex symmetric function} if it is convex, and if for any $\bx \in \R^n$ it holds that $g(\Pi \bx) = g(\bx)$ for all permutation matrices $\Pi \in \mathrm{Sym}(n)$.
\end{definition}
The properties of such functions are well-known in the literature on convex analysis and optimization, and they arise in many applications.  We briefly describe some of these properties and applications here.

An important class of convex symmetric functions is the set of linear functionals given by \emph{monotone linear functionals}:
\begin{equation*}
g(\bx) = \mathbf{v}^T \overline{\bx},
\end{equation*}
where $\mathbf{v}_1 \geq \cdots \geq \mathbf{v}_n$.  Recall that $\overline{\bx}$ is the vector obtained by sorting the entries of $\bx$ in descending order. Monotone linear functionals can be used to express any convex symmetric function.  Specifically, let $\mathcal{M} \subset \R^n$ represent the cone of monotone decreasing vectors in $\R^n$.  Then for any convex symmetric function $g: \R^n \rightarrow \R$, we have that
\begin{equation*}
g(\bx) = \sup_{\mathbf{v} \in \mathcal{M}} ~ \mathbf{v}^T \overline{\bx} - \alpha_{\mathbf{v}}.
\end{equation*}
This statement is a simple consequence of the separation theorem from convex analysis \cite{Roc1996}.  Monotone linear functionals in turn can be expressed as the nonnegative \emph{sum} of even more elementary functions called \emph{distribution functions}, which are defined as follows:
\begin{equation*}
g_k(\bx) = \sum_{i=1}^k ~ (\overline{\bx})_i.
\end{equation*}
These functions are closely related to the notion of conditional value-at-risk \cite{RocU2000}, which in turn is computed using
quantiles of probability distributions.

Convex symmetric functions are intimately connected with the concept of \emph{majorization} \cite{MarO1979}.  There are many equivalent characterizations of majorization \cite{Dav1957,Lew1995}, and we briefly mention some of these next.  A vector $\bx \in \R^n$ is said to majorize another vector $\by \in \R^n$ if
\begin{equation*}
g_k(\bx) \geq g_k(\by), ~ \forall k = 1 ,\dots, n-1 ~~~~ \mathrm{and} ~~~~ g_n(\bx) = g_n(\by).
\end{equation*}
The \emph{permutahedron} of a vector $\bx \in \R^n$ is the convex hull of all permutations of $\bx$, and is given by the set of vectors in $\R^n$ that are majorized by $\bx$.  Thus, convex constraints given by distribution functions provide a simple characterization of the permutahedron generated by $\bx$.  Majorization is also closely related to the notion of \emph{Lorenz dominance}; a (typically nonnegative) vector $\bx \in \R^n$ is said to Lorenz-dominate $\by \in \R^p$ if $-\bx$ is majorized by $-\by$.  Lorenz dominance is used to measure the level of inequality in distributions, i.e., if a distribution $\bx$ Lorenz-dominates a distribution $\by$ then $\bx$ is ``more equal'' than $\by$ (see also the Gini coefficient, which is used to measure inequalities in countries).

A convex symmetric function is an example of a \emph{Schur-convex function}, which is a function $f$ such that $f(\bx) \geq f(\by)$ whenever $\bx$ majorizes $\by$.  Hence a Schur-convex function preserves order with respect to majorization.  Consequently, such functions arise in many applications in which majorization plays a prominent role \cite{MarO1979}.  We note that the functions that are both convex and Schur-convex are exactly the convex symmetric functions.

A fairly similar set of results hold for convex functions of symmetric matrices that are invariant under conjugation of the argument by orthogonal matrices, i.e., convex functions $f: \Sn \rightarrow \R$ such that $f(V A V^T) = f(A)$ for all $A \in \Sn$ and for all $\Pi \in \mathrm{Sym}(n)$.



\end{document}